\documentclass[12pt,a4paper,oneside]{amsart}
\usepackage{amsmath,amsthm,amscd,amssymb,amsfonts, amsbsy, enumerate,vmargin}
\usepackage{latexsym}
\usepackage{txfonts}
\usepackage[colorlinks,citecolor=red,hypertexnames=false]{hyperref} 
\usepackage{pgf,tikz}
\usepackage{color}
\usepackage{boxedminipage}

\newfam\msbmfam 
\font\twelvemsbm=msbm10 scaled
\textfont\msbmfam\twelvemsbm\scriptfont\msbmfam\tenmsbm

\def\RR{{\mathbb R}} 
\def\CC{{\mathbb C}}

\def\div{{\rm div}}

\def\adj{\mathop{\rm adj}\nolimits}

\newcommand{\dint}{\iint}
\newcommand{\dist}{\operatorname{dist}}

\newcommand{\re}{\mathbb{R}}
\newcommand{\rn}{\mathbb{R}^n}

\newcommand{\reu}{\mathbb{R}^{n+1}_+}
\newcommand{\ree}{\mathbb{R}^{n+1}}

\newcommand{\NN}{\widetilde{N}_*}
\newcommand{\dd}{\mathbb{D}}
\newcommand{\vp}{\varphi}
\newcommand{\eps}{\varepsilon}
\newcommand{\epp}{\epsilon}
\newcommand{\s}{{\bf S}}
\newcommand{\pun}{\partial_\nu u}

\newcommand{\p}{p_{\alpha}}
\newcommand{\pp}{p^{+}}

\newcommand{\cS}{\mathcal{S}}

\newcommand{\tB}{\widetilde{B}}

\newcommand{\SL}{\mathcal{S}}
\newcommand{\D}{\mathcal{D}}
\newcommand{\lb}{\Lambda^\beta}
\newcommand{\la}{\Lambda^\alpha}
\newcommand{\pt}{\mathcal{P}_t}

\newcommand{\fiint}{\fint\!\!\!\!\fint}

\newcommand{\dv}{\operatorname{div}}

\DeclareMathOperator{\supp}{supp}

\newtheorem{theorem}{Theorem}[section] 
\newtheorem{proposition}[theorem]{Proposition}
\newtheorem{lemma}[theorem]{Lemma}
\newtheorem{corollary}[theorem]{Corollary}

\theoremstyle{definition}
\newtheorem{remark}[theorem]{Remark}

\title[The method of layer potentials]{The method of layer potentials in $L^p$ and endpoint spaces for elliptic operators with $L^\infty$ coefficients}
 
\author{Steve Hofmann, Marius Mitrea, and Andrew J. Morris}

\thanks{The work of the first two authors has been supported in part by NSF. The work of the third author has been supported by EPSRC grant EP/J010723/1.
}

\date{November 13, 2013}

\begin{document} 
\begin{abstract}
We consider layer potentials associated to elliptic operators $Lu=-\div\big(A \nabla u\big)$ acting in the upper half-space $\reu$ for $n\geq 2$, or more generally, in a Lipschitz graph domain, where the coefficient matrix $A$ is $L^\infty$ and $t$-independent, and solutions of $Lu=0$ satisfy interior estimates of De Giorgi/Nash/Moser type. A ``Calder\'on-Zygmund" theory is developed for the boundedness of layer potentials, whereby sharp $L^p$ and endpoint space bounds are deduced from $L^2$ bounds. Appropriate versions of the classical ``jump-relation" formulae are also derived.  The method of layer potentials is then used to establish well-posedness of boundary value problems for $L$ with data in $L^p$ and endpoint spaces.
\end{abstract}
 
\maketitle

\tableofcontents

\section{Introduction}
\setcounter{equation}{0}
Consider a second order, divergence form elliptic operator
\begin{equation}\label{L-def-1}
L=-{\rm div}\,A(x)\nabla\quad\mbox{in}\,\,\,
\RR^{n+1}:=\{X=(x,t):\,x\in\RR^n,\,\,t\in\RR\},
\end{equation}
where $A$ is an $(n+1)\times(n+1)$ matrix of $L^\infty$, $t$-independent, complex coefficients, satisfying the uniform ellipticity condition
\begin{equation}\label{eq1.1*}
\Lambda^{-1}|\xi|^{2}\leq\Re e\,\langle A(x)\,\xi,\xi\rangle
:= \Re e\sum_{i,j=1}^{n+1}A_{ij}(x)\,\xi_{j}\,\overline{\xi}_{i}, \quad
\Vert A\Vert_{\infty}\leq\Lambda,
\end{equation}
for some $\Lambda\in(0,\infty)$, for all $\xi\in\CC^{n+1}$, and for almost every $x\in\RR^{n}$. The operator $L$ is interpreted in the usual weak sense via the accretive sesquilinear form associated with~\eqref{eq1.1*}. In particular, we say that $u$ is a \textbf{``solution''} of $Lu=0$, or simply $Lu=0$, in a domain $\Omega\subset \RR^{n+1}$, if $u\in L^2_{1,\,\mathrm{loc}}(\Omega)$ and $\int_{\RR^{n+1}} A\nabla u \cdot \nabla \Phi =0$ for all $\Phi\in C^\infty_0(\Omega)$.

 Throughout the paper, we shall impose the following 
`` {\bf  standard assumptions}":
\begin{enumerate}
\item[(1)] The operator $L =-\dv A\nabla$ is of the type defined in \eqref{L-def-1} and \eqref{eq1.1*}
above, with 
$t$-independent  coefficient matrix $A(x,t)=A(x)$. 
\smallskip
\item[(2)] Solutions of $Lu=0$ satisfy the interior De Giorgi/Nash/Moser (DG/N/M) type
estimates defined in \eqref{eq2.DGN} and \eqref{eq2.M} below.
\end{enumerate}
The paper has two principal aims.   
First, we prove sharp $L^p$ and endpoint space bounds for layer potentials associated
to any operator $L$ that, along with its Hermitian adjoint $L^*$, satisfies the standard assumptions.
These results are of ``Calder\'on-Zygmund" type, in the sense that the $L^p$ and endpoint space bounds are deduced from $L^2$ bounds.  Second, we 
use the layer potential method to obtain well-posedness
results for boundary value problems for certain such $L$.
The precise definitions of the layer potentials, and a brief historical summary of previous
work (including the known $L^2$ bounds), is given below.

Let us now discuss certain preliminary matters needed to state our main theorems. For the sake of notational convenience, we will often use capital letters to denote points in $\ree$, e.g., $X=(x,t),\, Y=(y,s)$.  We let $B(X,r):= \{Y\in\ree:\,|X-Y|<r\}$, and $\Delta(x,r):= \{y\in\rn:\,|x-y|<r\}$ denote, respectively, balls of radius $r$ in $\ree$ and in $\rn$. We use the letter $Q$ to denote a generic cube in $\rn$, with sides parallel to the co-ordinate axes, and we let $\ell(Q)$ denote its side length. We adopt the convention whereby $C$ denotes a finite positive constant that may change from one line to the next but ultimately depends only on the relevant preceding hypotheses. We will often write $C_p$ to emphasize that such a constant depends on a specific parameter $p$. We may also write $a\lesssim b$ to denote $a\leq Cb$, and $a\approx b$ to denote $a\lesssim b \lesssim a$, for quantities $a,b\in\RR$.

\noindent{\bf De Giorgi/Nash/Moser (DG/N/M) estimates}.
We say that a locally square integrable function $u$ is ``locally
H\"{o}lder continuous", or equivalently, satisfies 
``De Giorgi/Nash (DG/N) estimates" in a domain $\Omega\subset\RR^{n+1}$, 
if there is a positive
constant $C_0<\infty$, and an exponent $\alpha\in(0,1]$, such that for any ball
$B=B(X,R)$
whose concentric double $2B := B(X,2R)$
is contained in $\Omega$, we have
\begin{equation}
|u(Y)-u(Z)|\leq
C_0\left(\frac{|Y-Z|}{R}\right)^{\alpha}\left(\fint_{2B}|u|^{2}\right)^{1/2},
\label{eq2.DGN}\end{equation}
whenever $Y,Z\in B$. 
Observe that any function $u$ satisfying \eqref{eq2.DGN} also 
satisfies Moser's ``local boundedness" estimate (see \cite{Mo})
\begin{equation}
\sup_{Y\in B}|u(Y)|\leq C_0 \left(\fint_{2B}|u|^{2}\right)^{1/2}.\label{eq2.M}\end{equation}
Moreover, as is well known, \eqref{eq2.M} self improves to
\begin{equation}\label{eq2.Mr}
\sup_{Y\in B}|u(Y)|\leq C_r \left(\fint_{2B}|u|^{r}\right)^{1/r},\qquad \forall \,r\in(0,\infty).
\end{equation}

\begin{remark}\label{firstrem}
It is well known (see \cite{DeG,Mo,Na}) that when the coefficient matrix $A$ is real, solutions of $Lu=0$ satisfy the DG/N/M estimates \eqref{eq2.DGN} and \eqref{eq2.M}, 
and the relevant constants depend quantitatively on ellipticity and dimension only
(for this result, the matrix $A$ need not be $t$-independent).
Moreover, estimate \eqref{eq2.DGN}, which implies \eqref{eq2.M}, is stable under small complex perturbations
of the coefficients in the $L^\infty$ norm (see, e.g., \cite[Chapter~VI]{Gi} or \cite{A}).  Therefore, the standard 
assumption (2) above holds automatically for small complex
perturbations of real symmetric elliptic coefficients.  We also note that in the $t$-independent setting 
considered here, 
the DG/N/M estimates always hold when the ambient dimension 
$n+1$ is equal to 3 (see \cite[Section 11]{AAAHK}).
\end{remark}

We shall refer to the following quantities collectively as the ``{\bf standard constants}":
the dimension $n$ in \eqref{L-def-1}, the ellipticity parameter $\Lambda$ in \eqref{eq1.1*}, and
the constants $C_0$ and $\alpha$ in the DG/N/M estimates \eqref{eq2.DGN} and \eqref{eq2.M}.

In the presence of DG/N/M estimates (for $L$ and $L^*$), by \cite{HK}, both
$L$ and $L^*$ have fundamental solutions $E: \{(X,Y)\in\RR^{n+1}\times \RR^{n+1} : X\neq Y\} \to \CC$ and $E^*(X,Y) := \overline {E(Y,X)}$, respectively, satisfying $E(X,\cdot),\,E(\cdot,X)\in L^2_{1,\,\mathrm{loc}}(\RR^{n+1}\setminus\{X\})$ and
\begin{equation}\label{eq7.4a}
L_{x,t} \,E (x,t,y,s) = \delta_{(y,s)},\quad
L^*_{y,s}\, E^*(y,s,x,t)= L^*_{y,s} \,\overline {E (x,t,y,s)} = \delta_{(x,t)},
\end{equation}
where $\delta_X$ denotes the Dirac mass at the point $X$. In particular, this means that 
\begin{equation}\label{fundsolprop}
\int_{\RR^{n+1}} A(x) \nabla_{x,t} E(x,t,y,s) \cdot \nabla \Phi(x,t)\, dxdt= \Phi(y,s), \qquad (y,s)\in\RR^{n+1}\,,
\end{equation}
for all $\Phi\in C^\infty_0(\RR^{n+1})$. Moreover, by the $t$-independence of our coefficients, 
\begin{equation*}
E(x,t,y,s)=E(x,t-s,y,0)\,.
\end{equation*}

As is customary, we then define the single and double layer potential 
operators, associated to $L$, in the upper and lower half-spaces, by
\begin{equation}\label{eq2.23}\begin{split}
\mathcal{S}^\pm f(x,t)&:=\int\limits_{\RR^n}E(x,t,y,0)\,f(y)\,dy,
\qquad (x,t)\in\RR^{n+1}_\pm\,,
\\[4pt]
\mathcal{D}^\pm f(x,t) & 
:=\int_{\mathbb{R}^{n}}\overline{\Big(\partial_{\nu^*} E^*
(\cdot,\cdot,x,t)\Big)}(y,0)\,f(y)\,dy,\qquad (x,t)\in\RR^{n+1}_\pm\,,
\end{split}
\end{equation}
where $\partial_{\nu^*}$ denotes the outer co-normal derivative
(with respect to $\reu$) associated to the adjoint matrix $A^*$,
i.e., 
\begin{equation}\label{eq2.23*}
\Big(\partial_{\nu^*} E^*(\cdot,\cdot,x,t)\Big)(y,0):= -e_{n+1}\cdot A^*(y)\,\Big(\nabla_{y,s} 
E^*(y,s,x,t)\Big)\big|_{s=0}
\,.
\end{equation}
Here,  $e_{n+1}:= (0,...,0,1)$ is the standard unit basis vector in the $t$ direction.  Similarly, 
using the notational convention that $t=x_{n+1}$, we define the outer co-normal derivative
with respect to $A$ by
$$\partial_{\nu} u:= -e_{n+1}\cdot A \nabla u=-\sum_{j=1}^{n+1}A_{n+1,j}\,\partial_{x_j} u\,.$$
When we are working in a particular half-space (usually the upper one, by convention), 
for simplicity of notation, we shall often drop the superscript and write simply, e.g.,
$\SL,\, \mathcal{D}$ in lieu of $\SL^+,\,\mathcal{D}^+$.  At times, it may be necessary to identify the operator $L$ to which the layer potentials are associated
(when this is not clear from context), in which case we shall 
write $\mathcal{S}_L,\,\mathcal{D}_L,$ and so on.  

We note at this point that for each fixed $t>0$ (or for that matter, $t<0$),
the operator $f\mapsto \mathcal{S}f(\cdot,t)$ is well-defined on $L^p(\rn)$, 
$1\leq p<\infty$, by virtue of the estimate 
\[
\int_{\rn}|E(x,t,y,0)| \, |f(y)| \, dy \,\lesssim \, I_1(|f|)(x)\,,
\]
which follows from~\eqref{G-est1} below, where $I_1$ denotes the classical Riesz potential.
Also, the operator $f\mapsto \mathcal{D}f(\cdot,t)$ 
is well-defined on $L^p(\rn)$, $2-\eps<p<\infty$, by virtue of the estimate
\[
\int_{\rn}\big|\Big(\nabla E(x,t,\cdot,\cdot)\Big)(y,0)\big|^q \, dy \,\lesssim \, t^{n(1-q)}\,,\quad 1<q<2+\eps\,,
\]
which follows from \cite[Lemmata 2.5 and 2.8]{AAAHK} (see also \cite[Proposition 2.1]{AAAHK}, which guarantees that $\nabla E$ makes sense on horizontal slices in the first place).

We denote the boundary trace of the single layer potential by
 \begin{equation}\label{eq2.24}
S\!f(x):=\int\limits_{\RR^n}E(x,0,y,0)\,f(y)\,dy,\qquad x\in\RR^n,
\end{equation}
which is well-defined on $L^p(\rn)$, $1\leq p<\infty$, by~\eqref{G-est1} below. We shall also define, in Section~\ref{s2} below, boundary singular integrals
\begin{equation}
\begin{split}\label{eq1.6}Kf(x) & :=``\mathrm{p.v.}"\int_{\mathbb{R}^{n}}\overline{\Big(\partial_{\nu^*} E^* (\cdot,\cdot,x,0)\Big)}(y,0)\,f(y)\,dy\,,\\[4pt]
\widetilde{K}f(x) & := ``\mathrm{p.v.}"\int_{\mathbb{R}^{n}}\Big(\partial_\nu E(x,0,\cdot,\cdot)\Big)(y,0)\,f(y)\,dy\,,\\[4pt]
\mathbf{T} f(x) & := ``\mathrm{p.v.}\,"\int_{\RR^n} \Big(\nabla E\Big)(x,0,y,0)f(y)\,dy\,,
\end{split}
\end{equation}
where the ``principal value" is purely formal, since we do not actually establish convergence of a principal value.  
We shall give precise definitions and derive the jump relations for the layer potentials in Section \ref{s2}. Classically,
$\widetilde{K}$ is often denoted $K^{\ast}$, but we avoid this notation here, as $\widetilde{K}$ need not be the adjoint of $K$ unless
$L$ is self-adjoint.   In fact,  using the notation $\adj(T)$ to denote the Hermitian adjoint of an operator $T$ acting in $\mathbb{R}^n$, we have that $\widetilde{K}_L = \adj (K_{L^*})$.

Let us now recall
the definitions of the non-tangential maximal operators $N_{\ast},\widetilde{N}_{\ast}$,
and of the notion of ``non-tangential convergence". 
Given $x_0\in\mathbb{R}^{n}$, define
the cone $\Gamma(x_0):=\{(x,t)\in\mathbb{R}_{+}^{n+1}:|x_0-x|<t\}$. Then for measurable functions $F:\mathbb{R}_{+}^{n+1}\rightarrow\CC$, define
\begin{equation*}
\begin{split} N_{\ast} F(x_0)& :=\sup_{(x,t)\in\Gamma(x_0)}|F(x,t)|,\\[4pt]
\widetilde{N}_{\ast} F(x_0) & :=\sup_{(x,t)\in \Gamma(x_0)}\left(\fint\!\!\fint_{|(x,t)-(y,s)|<t/4}
|F(y,s)|^{2}dyds\right)^{1/2}\,,
\end{split}\end{equation*}
where $\fint_E f := |E|^{-1} \int_E f$ denotes the mean value. We shall say that  $F$ ``converges non-tangentially"
to a function $f:\rn\rightarrow\CC$, and write $F\overset{{\rm n.t.}}{\longrightarrow} f$, 
if for a.e. $x\in \rn$,
$$\lim\limits_{\Gamma(x)\ni(y,t)\to(x,0)}F(y,t)=f(x)\,.$$
These definitions have obvious analogues in the lower half-space $\mathbb{R}_{-}^{n+1}$ that we distinguish by writing $\Gamma^\pm, N^\pm_{\ast},\widetilde{N}^\pm_{\ast}$, e.g., the cone $\Gamma^-(x_0):=\{(x,t)\in\mathbb{R}_{-}^{n+1}:|x_0-x|<-t\}$.

As usual, for $1<p<\infty$, let $\dot{L}_1^p(\RR^n)$ denote the homogenous Sobolev space of order one, which is defined as the completion of $C_0^\infty(\rn)$, with respect to the norm $\|f\|_{\dot{L}^p_1}:=\|\nabla f\|_p$, realized as a subspace of the space $L^1_{\mathrm{loc}}(\RR^n)/\CC$ of locally integrable functions modulo constant functions.

As usual, for $0<p\leq1$, let $H^p_{\mathrm{at}}(\rn)$ denote the classical atomic Hardy space, which is a subspace of the space $\s'(\RR^n)$ of tempered distributions (see, e.g., \cite[Chapter III]{St2} for a precise definition). Also, for $n/(n+1)<p\leq1$, let $\dot{H}^{1,p}_{\mathrm{at}}(\rn)$ denote the homogeneous ``Hardy-Sobolev" space of order one, which is a subspace of $\s'(\RR^n)/\CC$ (see, e.g., \cite[Section~3]{MM} for further details). In particular, we call $a\in \dot{L}^2_1(\rn)$ a {\it regular atom} if there exists a cube $Q\subset\rn$ such that 
\begin{equation*}
\supp a\subset Q,\qquad
\|\nabla a\|_{L^{2}(Q)}\leq 
|Q|^{\frac{1}{2}-\frac{1}{p}}, 
\end{equation*}
and we define the space
\[
\dot{H}^{1,p}_{\mathrm{at}}(\rn) := \{{f\in\s'(\RR^n)/\CC} : \nabla f=\sum_{j=1}^\infty
\lambda_j\nabla a_j,\,\,(\lambda_j)_j\in\ell^p,\,\,
a_j\,\,\mbox{is a regular atom}\}, 
\]
where the series converges in $H^p_{\mathrm{at}}(\rn)$, and the space is equipped with the quasi-norm $\|f\|_{\dot{H}^{1,p}_{\mathrm{at}}(\rn)}:={\rm inf}\,[\sum_j |\lambda_j|^p]^{1/p}$, where the infimum is taken over all such representations.

We now define the scales 
\[
H^p(\RR^n):=\left\{
\begin{array}{l}
H^p_{\mathrm{at}}(\RR^n)\,,
\,\,0<p\leq 1,
\\[6pt]
L^p(\RR^n)\,,
\,\,1<p<\infty,
\end{array}
\right. \, \quad 
\dot{H}^{1,p}(\RR^n):=\left\{
\begin{array}{l}
\dot{H}^{1,p}_{\mathrm{at}}(\RR^n)\,,
\,\,\frac{n}{n+1}<p\leq 1,
\\[6pt]
\dot{L}_1^p(\RR^n)\,,
\,\,1<p<\infty.
\end{array}\right.
\]

We recall that, by the classical result of C. Fefferman (cf. \cite{FS}), 
the dual of $H^1(\rn)$ is $BMO(\rn)$.
Moreover, 
$(H^p_{\mathrm{at}})^*=\dot{C}^{\alpha}(\rn)$,  if  $\alpha:=n(1/p-1)\in (0,1)$,
where $\dot{C}^{\alpha}(\rn)$
denotes  the homogeneous H\"older space of order $\alpha$.
In general,
for a measurable set $E$, and for $0<\alpha<1$, the H\"{o}lder space $\dot{C}^\alpha (E)$
is defined to be the set 
of $f\in C(E)/ \CC$ satisfying
\[
\|f\|_{\dot{C}^\alpha}:= \sup \frac{|f(x)-f(y)|}{|x-y|^\alpha}\,<\,\infty,
\]
where the supremum is taken over all pairs $(x,y)\in E\times E $ such that $x\neq y$.
For $0\leq \alpha<1$, we define the scale
\[
\Lambda^\alpha(\RR^n):=\left\{
\begin{array}{l}
\dot{C}^{\alpha}(\rn)\,, \,\,0<\alpha < 1,
\\[6pt]
BMO(\RR^n)\,,\,\,\alpha =0\,.
\end{array}
\right.
\]

As usual, we say that  a function $F\in L^2_{\mathrm{loc}}(\reu)$ belongs to the tent space
$T^\infty_2(\reu)$, if it satisfies the Carleson measure condition
\[
\|F\|_{T^\infty_2(\reu)}:=
\left(\sup_Q\frac1{|Q|}\iint_{R_Q}|F(x,t)|^2\, \frac{dx dt}{t}\right)^{1/2}\,<\,\infty\,.
\]
Here, the supremum is taken over all cubes $Q\subset\RR^n$, and
$R_Q:=Q\times(0,\ell(Q))$ is the usual ``Carleson box" above $Q$.

With these definitions and notational conventions in place,
we are ready to state the first main result of this paper.  

\begin{theorem}\label{P-bdd1}
Suppose that $L$ and $L^*$ satisfy the standard assumptions, let $\alpha$ denote the minimum of the De Giorgi/Nash exponents for $L$ and $L^*$ in \eqref{eq2.DGN}, and set ${\p:=n/(n+\alpha)}$.
Then there exists $\pp>2$, depending only on the standard constants,
such that
\begin{align}
\sup_{t> 0}\|\nabla \mathcal{S} f(\cdot,t)\|_{L^p(\RR^n,\CC^{n+1})} & \leq  C_p\|f\|_{H^p(\RR^n)},
\qquad\forall\,p\in(\p,\pp)\,, 
\label{LP-2}
\\[4pt]
\|\widetilde{N}_*\left(\nabla{\mathcal{S}}f\right)\|_{L^p(\RR^n)} & \leq  C_p\|f\|_{H^p(\RR^n)},
\qquad\forall\,p\in(\p,\pp)\,, 
\label{LP-1}
\\[4pt]
\|\nabla_xS\!f\|_{H^{p}(\RR^n,\CC^n)}  & \leq  C_p\|f\|_{H^p(\RR^n)}, 
\qquad\forall\,p\in(\p,\pp)\,,
\label{LP-3a}
\\[4pt] 
\|\widetilde{K}f\|_{H^p(\RR^n)} & \leq  C_p\|f\|_{H^p(\RR^n)}, 
\qquad\forall\,p\in(\p,\pp)\,,
\label{LP-3}
\\[4pt]
\|N_*(\mathcal{D} f)\|_{L^p(\RR^n)} & \leq  C_p\|f\|_{L^p(\RR^n)}, 
\qquad\forall\,p\in\left(\frac{\pp}{\pp-1},\infty\right)\,,
\label{LP-4}
\\[4pt]
\|t\nabla \mathcal{D}f\|_{T^\infty_2(\reu)}& \leq  C\|f\|_{BMO(\RR^n)}\,, 
\label{LP-5}
\\[4pt]
\|\mathcal{D} f\|_{\dot{C}^\beta(\overline{\reu})}& \leq  {C_\beta}\|f\|_{\dot{C}^\beta(\rn)}\,, 
\qquad\forall\,\beta\in\left(0,\alpha\right)\,,
\label{LP-6}
\end{align}
for an extension of $\mathcal{D} f$ to $\overline{\reu}$, where $S$, $\mathcal{S}$, $\mathcal{D}$, and $\widetilde{K}$ may correspond to either $L$ or $L^*$, and the analogous bounds hold in the lower half-space.
\end{theorem}

To state our second main result, let us recall the definitions of the
 Neumann and  Regularity problems, with (for now) $n/(n+1)<p<\infty$:
\[
(\mbox{N})_p\,\,\left\{
\begin{array}{l}
Lu=0\mbox{ in }\reu,\\[6pt]
\widetilde{N}_*(\nabla u)\in L^p(\rn),\\[6pt]
\partial_\nu u(\cdot,0)=g\in H^p(\rn),
\end{array}\right.
\qquad
(\mbox{R})_p\,\,\left\{
\begin{array}{l}
Lu=0\mbox{ in }\reu,\\[6pt]
\widetilde{N}_*(\nabla u)\in L^p(\rn),\\[6pt]
u(\cdot,0)=f\in \dot{H}^{1,p}(\rn)\,,
\end{array}\right. 
\]
where we specify that the solution $u$, of either
$(N)_p$ or  $(R)_p$, will assume its boundary data in the
the following sense:
\begin{itemize}
\item $u(\cdot ,0)\in \dot{H}^{1,p}(\rn)$, 
and $u\overset{{\rm n.t.}}{\longrightarrow} u(\cdot,0)$;

\smallskip

\item $\nabla_x u(\cdot,0)$ and $\partial_\nu u(\cdot,0)$ belong to $H^p(\rn)$,
and are the weak limits, 
(in $L^p$, for $p>1$, and in the
sense of tempered distributions, if $p\leq 1$), 
as $t\to 0$, of
$\nabla_x u(\cdot,t)$, and of $-e_{n+1}\cdot A\nabla u(\cdot,t)$,
respectively. 
\end{itemize}
We also formulate 
the Dirichlet problem in $L^p$, with $1<p<\infty$:
\[
(\mbox{D})_p\,\,\left\{
\begin{array}{l}
Lu=0\mbox{ in }\reu,\\[6pt]
N_*(u)\in L^p(\rn),\\[6pt]
u(\cdot,0)=f\in L^p(\rn)\,,
\end{array}\right.
\]
and in $\Lambda^\alpha,$ with $0\leq\alpha<1$:
\[
(\mbox{D})_{\Lambda^\alpha}\,\,\left\{
\begin{array}{l}
Lu=0\mbox{ in }\reu,\\[6pt]
t\nabla u \in T^\infty_2(\reu)\,\,{\rm if}\,\, \alpha =0\,,\,\, {\rm or}\,\, u\in \dot{C}^\alpha(\overline{\reu})\,\,
{\rm if}\,\,0<\alpha<1\,,\\[6pt]
u(\cdot,0)=f\in \Lambda^\alpha(\rn)\,\,.
\end{array}\right.
\]
The solution $u$ of $(D)_p$, with data $f$, satisfies
\begin{itemize}
\item $u\overset{{\rm n.t.}}{\longrightarrow} f$, and $u(\cdot,t)\to f$ as $t\to 0$ in $L^p(\rn)$.
\end{itemize}
The solution $u$ of $(D)_{\la}$, with data $f$, satisfies
\begin{itemize}
\item  $u(\cdot,t)\to f$ as $t\to 0$
in the weak* topology on $\la,\, 0\leq \alpha <1$. 

\smallskip

\item $u \in \dot{C}^\alpha(\overline{\reu})$, and $u(\cdot,0) = f$ pointwise, $0<\alpha<1$.
\end{itemize}

\begin{theorem}\label{th2}  Let $L=-\dv A\nabla$ and $L_0=-\dv A_0\nabla$ be 
 as in \eqref{L-def-1} and \eqref{eq1.1*} with $A=A(x)$ and $A_0=A_0(x)$  
both $t$-independent, and suppose that $A_0$ is real symmetric. There exists $\eps_0>0$ and $\epsilon>0$, both depending only on dimension and the ellipticity of $A_0$, such that if
\[
\|A-A_0\|_\infty < \eps_0\,,
\]
then $(N)_p$, $(R)_p$, $(D)_q$ and $(D)_{\la}$ are uniquely solvable for $L$ and $L^*$ when $1-\epp<p<2+\epp$, $(2+\epp)'<q<\infty$ and $0\leq\alpha <n\epp/(1-\epp)$,
respectively.
\end{theorem}

\begin{remark}
By Remark~\ref{firstrem}, both $L$ and $L^*$ satisfy the ``standard assumptions" under the hypotheses of Theorem \ref{th2}. 
\end{remark}

\begin{remark}
Theorems \ref{P-bdd1} and \ref{th2} continue to hold,
with the half-space $\reu$ replaced by a Lipschitz graph domain of the form
$\Omega=\{(x,t)\in\ree: \, t>\phi(x)\}$, where $\phi:\rn \to \re$ is a Lipschitz function.
Indeed, that case may be reduced to that of the half-space by a standard pull-back mechanism.
We omit the details. 
\end{remark}

\begin{remark}
In the case of the Dirichlet problem with $\dot{C}^\alpha$ data, we answer in the affirmative a higher dimensional version of a question posed in $\RR^2_+$ in \cite{SV}. We note that this particular result is new even in the case $A=A_0$ (i.e., in the case that $A$
is real symmetric).
\end{remark}

Let us briefly review some related history.  We focus first on the question of boundedness of 
layer potentials.  As we have just noted,
our results extended immediately to the setting of a Lipschitz graph domain.
The prototypical result in that setting is the result of
Coifman, McIntosh and 
Meyer~\cite{CMM} concerning the $L^{2}$ boundedness
of the Cauchy integral operator on a Lipschitz curve,
which implies $L^2$ bounds for the layer potentials associated to the Laplacian 
via the method of rotations.
In turn, the corresponding $H^p/L^p$ bounds follow by classical Calder\'on-Zygmund theory.

For the  variable coefficient operators considered here, 
the $L^2$ boundedness theory 
(essentially, the case $p=2$ of Theorem \ref{P-bdd1}, along with $L^2$ square function estimates)
was introduced in 
\cite{AAAHK}. In that paper, it was shown, first, that such $L^2$ bounds (along with $L^2$ invertibility for
$\pm(1/2)I+K$) are stable under
small complex $L^\infty$ perturbations of the coefficient matrix, and second,
that these boundedness and invertibility results hold in the case that $A$ is real and symmetric (hence also for complex perturbations
of  real symmetric $A$).  The case $p=2$ for $A$  real, but not necessarily symmetric,
was treated in \cite{KR} in the case $n=1$ (i.e., in ambient dimension $n+1 = 2$), 
and in \cite{HKMP1}, in all dimensions.  
Moreover, in hindsight, in the special 
case that the matrix $A$ is of the
``block" form
\[
\left[\begin{array}{c|c}
 & 0\\ B & \vdots\\
 & 0\\
\hline 0\cdots0 & 1\end{array}\right],
\]
 where $B=B(x)$ is a $n\times n$ matrix, $L^2$ bounds for layer potentials follow 
 from the solution of the Kato problem
\cite{AHLMcT},
 since in the block case the single layer potential is given by
 $\SL f(\cdot,t)= (1/2) J^{-1/2} e^{-t\sqrt{J}},$
 where $J:= -\dv_x B(x)\nabla_x$.

 Quite recently, the case $p=2$ of Theorem \ref{P-bdd1} was shown to hold in general,
 for $L$ and $L^*$ satisfying the ``standard assumptions",
 in work of Rosen \cite{R}, in which $L^2$ bounds for layer potentials are obtained
 via results of \cite{AAM} concerning functional calculus of certain 
 first order ``Dirac-type" operators\footnote{
 A direct proof of these $L^2$ bounds for layer potentials, 
 bypassing the functional calculus results of \cite{AAM},
 will appear in \cite{GH}.}.  We note further that Rosen's $L^2$ estimates do not require the
 DG/N/M hypothesis (rather, just ellipticity and $t$-independence).  
 On the other hand, specializing to the ``block" case mentioned above,
 we observe that counter-examples in \cite{MNP} and \cite{Frehse} 
(along with some observations in \cite{AuscherSurvey}), show that the 
full range of $L^p$ and Hardy space  results treated in the present paper 
cannot be obtained without assuming DG/N/M.
It seems very likely that $L^p$ boundedness for some restricted range of $p$ should still
hold, even in the absence of DG/N/M, as is true in the special case of the ``block matrices" 
treated in \cite{AuscherSurvey}, \cite{BK}, \cite{HoMa}, 
and \cite{May}, 
but we have not considered this question here.
 We mention also that even in the presence of DG/N/M (in fact, even for $A$ real and symmetric), 
the constraint on the
upper bound on $p$ in \eqref{LP-2}-\eqref{LP-1}
is optimal.  To see this, consider the block case, so that $L$ is of the form
$Lu=u_{tt} +\div_x B(x)\nabla_xu=: u_{tt}-Ju$, 
where $B=B(x)$ is an $n\times n$ uniformly elliptic matrix.  Thus, $\cS f(\cdot,t)= (1/2) J^{-1/2}e^{-t\sqrt{J}} f$, so that, considering only the tangential part of the gradient in \eqref{LP-2}, and letting $t\to 0$, we obtain as a consequence of  \eqref{LP-2} that
\begin{equation}\label{eq4.68}
\|\nabla_xJ^{-1/2} f\|_p \lesssim \|f\|_p\,.
\end{equation}
But by Kenig's examples (see \cite[pp.~119--120]{AT}), for each $p>2$, there is a $J$ as above
for which the Riesz transform bound \eqref{eq4.68} fails.  The matrix $B$  may even be taken to be real symmetric.
    Thus, our results are in the nature of best possible, in the sense that, first, the DG/N/M hypothesis is needed to treat $p$ near (or below) 1, and second,
 that even with DG/N/M, the exponent $\pp$ is optimal.

As regards the question of solvability, addressed here in Theorem \ref{th2}, we recall that
in the special case of the Laplacian on a Lipschitz domain,  solvability of the $L^p$
Dirichlet problem is due to Dahlberg \cite{D}, while the Neumann and Regularity problems
were treated first, in the case $p=2$, by Jerison and Kenig  \cite{JK2}, and then by Verchota
\cite{V}, by an alternative proof using the method of layer potentials; and second, 
in the  case $1<p< 2+\eps$,
by Verchota \cite{V} (Regularity problem only), and in the case $1\leq p< 2+\eps$  
by Dahlberg and Kenig \cite{DK} (Neumann and Regularity), and finally, 
in the case $1-\eps<p<1$ by Brown \cite{Br} (who then obtained $D_{\la}$ by duality).  
A conceptually different proof of the latter result has been 
subsequently given by Kalton and Mitrea in \cite{KM} using a general perturbation 
technique of functional analytic nature\footnote{Thus answering a question 
posed by E. Fabes.}.
More generally, in the setting of variable coefficients,
in the special case that $A=A_0$ (i.e., that $A$ is real symmetric), the $L^p$ results for the Dirichlet problem were obtained by Jerison and Kenig \cite{JK1}, and for the Neumann and Regularity problems
by Kenig and Pipher in \cite{KP} (the latter authors also treated the analogous Hardy space theory 
in the case $p=1$).   
The case $p=2$  of Theorem \ref{th2} 
(allowing complex coefficients) was obtained first in \cite{AAAHK}, with an alternative proof
given in \cite{AAH}.  The case $n=1$ (i.e., in ambient dimension $n+1=2$) of Theorem
\ref{th2} follows from the work of Barton \cite{Bar}. 
  
In the present work, we consider solvability of boundary value problems
only for complex perturbations of real, symmetric operators, but
we point out that there has also been some recent progress in the case of non-symmetric $t$-independent operators.  For real, non-symmetric coefficients, 
the case $n=1$ has been treated by Kenig, Koch, Pipher and Toro
\cite{KKPT} (Dirichlet problem), and by Kenig and Rule \cite{KR} 
(Neumann and Regularity). 
The work of Barton
\cite{Bar} allows for
complex perturbations of the results of \cite{KKPT} and \cite{KR}.
The higher dimensional case
$n>1$ has very recently been treated in \cite {HKMP1} (the Dirichlet
problem for real, non-symmetric operators),
and in 
\cite{HKMP2} (Dirichlet and Regularity, for 
complex perturbations of the real, non-symmetric case).
In these results for non-symmetric operators, 
necessarily there are additional restrictions on the range of allowable $p$, as
compared to the symmetric case 
(cf. \cite{KKPT}).  We remark that in the non-symmetric setting, with $n>1$, 
the Neumann problem remains open.

We mention that we have also obtained an analogue of Theorem \ref{th2} for the Transmission
problem, which we plan to present in a forthcoming publication \cite{HMM}.

 Finally, let us discuss briefly the role of $t$-independence in our ``standard assumptions".
 Caffarelli, Fabes and Kenig \cite{CFK} have shown that some
regularity, in a direction transverse to the boundary, is needed to obtain $L^p$ solvability for,
say, the Dirichlet problem.   Motivated by their work, one may naturally split
the theory of boundary value problems for elliptic operators
in the half-space\footnote{There are analogues
of the theory in a star-like Lipschitz domain.}  into two parts:
1) solvability theory for $t$-independent operators, and 2) solvability results in which the discrepancy
$|A(x,t) -A(x,0)|$, which measures regularity in $t$ at the boundary, is controlled by a Carleson measure estimate of the type considered in \cite{FKP}\footnote{The Carleson measure
control of \cite{FKP} is essentially optimal, in view of \cite{CFK}.}, and in which one has some good solvability result for the operator with $t$-independent coefficients $A_0(x):=A(x,0)$. 
The present paper, and its companion article \cite{HMM}, fall into category 1).  The paper \cite{HMaMo} falls into category 2), and uses our results
here to obtain boundedness and solvability results for operators
in that category, in which the Carleson measure estimate for the discrepancy is sufficiently small
(in this connection, see also the previous work \cite{AA}, which treats the case $p=2$).

\vskip 0.08in
\noindent{\it Acknowledgments.} The first named author thanks S. Mayboroda for suggesting a 
simplified proof of estimate \eqref{LP-4}.   The proof of item (vi) of Corollary \ref{cor4.47} arose
in discussions between the first author and M. Mourgoglou.

\section{Jump relations and definition of the boundary integrals}\label{s2}
\setcounter{equation}{0}
Throughout this section, we impose the ``standard assumptions" defined previously. The operators ${\rm div}$ and $\nabla$ are considered in all $n+1$ 
variables, and we write ${\rm div}_x$ and $\nabla_x$ when only the first $n$ variables are 
involved. Also, since we shall consider operators $T$ that may be viewed as acting either
in $\RR^{n+1}$, or in $\RR^n$ with the $t$ variable frozen, we need to distinguish Hermitian 
adjoints in these two settings.  We therefore use $T^*$ to denote the $(n+1)$-dimensional 
adjoint of $T$, while $\adj(T)$ denotes the adjoint of $T$ acting in $\RR^n$.

As usual, to apply the layer potential method, we shall need to understand the jump relations for the co-normal 
derivatives of $u^\pm= \mathcal{S}^\pm f.$  To this end, 
let us begin by recording the fact that, by the main result of \cite{R},
\begin{equation}\label{eq4.prop46}
\sup_{\pm t> 0}\|\nabla \mathcal{S}_L^\pm\,f(\cdot,t)\|_{L^2(\RR^n,\CC^{n+1})}+\sup_{\pm t> 0}
\|\nabla \mathcal{S}_{L^*}^\pm\,f(\cdot,t)\|_{L^2(\RR^n,\CC^{n+1})}\leq C\|f\|_{L^2(\RR^n)}\,.
\end{equation}
Combining the last estimate with \cite[Lemma 4.8]{AAAHK} (see Lemma~\ref{l4.8*} below), we obtain
\begin{equation}\label{eq4.9}
\|\widetilde{N}_*^\pm(\nabla \cS^\pm f)\|_{L^2(\RR^n)}\leq C\|f\|_{L^2(\RR^n)}\,.
\end{equation}

Next, we recall the following fact proved in \cite{AAAHK}. Recall that $e_{n+1} :=(0,...,0,1)$ denotes the standard unit basis vector in the $t=x_{n+1}$ direction.
\begin{lemma}[{\cite[Lemmata 4.1 and 4.3]{AAAHK}}]\label{l4.1}
Suppose that $L$ and $L^*$ satisfy the standard assumptions. If $Lu=0$ in $\RR_{\pm}^{n+1}$ and $\widetilde{N}_*^\pm(\nabla u) \in L^2(\RR^n)$, then the co-normal derivative $\partial_\nu u(\cdot,0)$ exists in the variational sense and belongs to $L^2(\RR^n)$,  i.e., there exists a unique $g\in L^2(\RR^n)$, and we set $\partial_\nu u(\cdot,0):=g$, with $\|g\|_2\leq C\|\widetilde{N}_*^\pm(\nabla u)\|_2$, such that
\begin{enumerate}
\item[(i)] $\int_{\RR_{\pm}^{n+1}}A\nabla u\cdot \nabla\Phi\, dX = \pm \int_{\RR^n} g\, \Phi(\cdot,0) \, dx\,$ for all $\Phi \in C^\infty_0(\ree)$.
\item[(ii)] $-\langle A\nabla u(\cdot,t),e_{n+1}\rangle \to g$ weakly in $L^2(\rn)$ as $t\to 0^\pm$.
\end{enumerate}
Moreover, there exists a unique $f\in \dot{L}_1^2(\rn)$ , with $\|f\|_{\dot{L}^2_1(\rn)}\leq C\|\widetilde{N}_*^\pm(\nabla u)\|_2$, such that
\begin{enumerate}
\item[(iii)] $u\to f$ non-tangentially.
\item[(iv)] $\nabla_xu(\cdot,t)\to \nabla_x f$ weakly in $L^2(\rn)$ as $t\to 0^\pm$.
\end{enumerate}
\end{lemma}

For each $f\in L^2(\RR^n)$, it follows from \eqref{eq7.4a} and \eqref{eq4.prop46} that $u:=\mathcal{S}^{\pm}f$ is a solution of $Lu=0$ in $\RR^{n+1}_\pm$, and this solution has the properties listed in Lemma~\ref{l4.1} because \eqref{eq4.9} holds. We then have the following result. 

\begin{lemma}\label{L-jump} 
Suppose that $L$ and $L^*$ satisfy the standard assumptions. If $f\in L^2(\RR^n)$, then almost everywhere on $\RR^n$, we have 
\begin{equation}\label{jump-1}
\partial_\nu\mathcal{S}^+f(\cdot,0) - \partial_\nu\mathcal{S}^-f(\cdot,0) = f,
\end{equation}
where the co-normal derivatives are defined in the variational sense of Lemma \ref{l4.1}.
\end{lemma}

\begin{proof}
Let us first suppose that $f\in C^\infty_0(\RR^n)$, and introduce 
\[
u:=\left\{
\begin{array}{l}
{\mathcal{S}}^+f\,\,\,\mbox{in}\,\,\,\RR^{n+1}_+,
\\[4pt]
{\mathcal{S}}^-f\,\,\,\mbox{in}\,\,\,\RR^{n+1}_-,
\end{array}
\right.
\]
and pick some $\Phi\in C^\infty_0(\RR^{n+1})$. 
By Lemma \ref{l4.1} (i) and the property
of the fundamental solution in \eqref{fundsolprop}, we obtain
\begin{align}\begin{split}\label{IBP-1}
\int_{\RR^n}\Big\{\partial_\nu\mathcal{S}^+f(x,0)& - \partial_\nu\mathcal{S}^-f(x,0) \Big\}\,\Phi(x,0)\,dx
\\[4pt]
&
=\int_{\RR^n}\partial_\nu\mathcal{S}^+f(x,0) \,\Phi(x,0)\,dx
\,-\,\int_{\RR^n} \partial_\nu\mathcal{S}^-f(x,0)\,\Phi(x,0)\,dx
\\[4pt]
&
=\int_{\RR^{n+1}_{+}}\langle A\nabla u,\nabla\Phi\rangle\,dX
\,+\,\int_{\RR^{n+1}_{-}}\langle A\nabla u,\nabla\Phi\rangle\,dX
\\[4pt]
&
=\int_{\RR^{n+1}}\Bigl\langle A(x)\left(
\int_{\RR^n}\nabla_{x,t}E(x,t,y,0)f(y)\,dy\right)\,,\,\nabla\Phi(x,t)
\Bigr\rangle\,dxdt
\\[4pt]
&
=\int_{\RR^n}f(y)\left(\int_{\RR^{n+1}}\langle A(x)\nabla_{x,t}E(x,t,y,0),
\nabla\Phi(x,t)\rangle\,dxdt\right)\,dy
\\[4pt]
&
=\int_{\RR^n}f(y)\,\Phi(y,0)\,dy.
\end{split}\end{align}
The use of Fubini's theorem in the fifth line is justified by absolute convergence, since
$\nabla E(\cdot,Y) \in L^p_{\mathrm{loc}}(\ree),\, 1\leq p<(n+1)/n$ (cf. \cite[Theorem 3.1]{HK}).

Given an arbitrary $f\in L^2(\RR^n)$, we may approximate $f$ 
by $f_k\in C^\infty_0(\RR^n)$, and observe that both the first and last lines in \eqref{IBP-1} converge appropriately (for the first line, this follows from
\eqref{eq4.9} and Lemma \ref{l4.1}). Then, since $\Phi$ was arbitrary in $C^\infty_0(\ree)$, (\ref{jump-1}) follows. 
\end{proof}

In view of (\ref{jump-1}), we now define the bounded operators $K,\, \widetilde{K}:L^2(\RR^n)\rightarrow L^2(\RR^n)$ and $\mathbf{T}: L^2(\RR^n)\to  L^2(\RR^n,\CC^{n+1})$, as discussed in \eqref{eq1.6}, rigorously by
\begin{align*}
\widetilde{K}_Lf&:=-{\textstyle{\frac{1}{2}}}f + \partial_\nu\mathcal{S}^+_Lf(\cdot,0)
={\textstyle{\frac{1}{2}}}f + \partial_\nu\mathcal{S}^-_Lf(\cdot,0)\\[4pt]
K_L f&:= \adj(\widetilde{K}_{L^*}) f\\[4pt]
\mathbf{T}_Lf &:= \Big(\nabla_x S_{\!L} f \,,\, \tfrac{-1}{A_{n+1,n+1}}\big(\widetilde{K}_{L}f + \textstyle{\sum_{j=1}^n A_{n+1,j}}\, \partial_{x_j}S_{\!L} f\big) \Big).
\end{align*}

We then have the following lemma, which we quote without proof  
from \cite{AAAHK}\footnote{\cite[Lemma 4.18]{AAAHK} assumes that \eqref{eq4.prop46} holds,
but as noted above,  it is now known that this is always the case, given our standard assumptions,
by the result of \cite{R}.}, although part (i) below is just a rephrasing of Lemma~\ref{l4.1}(ii) and Lemma~\ref{L-jump}.
\begin{lemma}[{\cite[Lemma 4.18]{AAAHK}}]\label{l4.ntjump} 
Suppose that $L$ and $L^*$ satisfy the standard assumptions. If $f\in L^2(\mathbb{R}^n)$, then
\begin{enumerate}
\item[(i)] $\partial_\nu(\mathcal{S}^\pm_Lf)(\cdot,0) = \left(\pm \frac{1}{2}I + \widetilde{K}_L\right)f$,\\
and $-\langle A\nabla{\mathcal{S}}_L^\pm f(\cdot, t),e_{n+1}\rangle \to \left(\pm \frac{1}{2}I + \widetilde{K}_L\right)f$ weakly in $L^2$ as $t\to 0^\pm$,
\end{enumerate}
where the co-normal derivative is defined in the variational sense of Lemma \ref{l4.1}.
\begin{enumerate}
\item[(ii)]  $\nabla \mathcal{S}^\pm_L f(\cdot,t) 
\to \left(\mp\frac{1}{2A_{n+1,n+1}}e_{n+1} + {\bf T}_L\right)f$ weakly in $L^2$ as $t\to 0^\pm$,
\end{enumerate}
where the tangential component of ${\bf T}_Lf$ equals $\nabla_x S_{\!L}f$.
\begin{enumerate}
\item[(iii)]$\mathcal{D}_L^{\pm }f (\cdot,t) \to \left(\mp\frac{1}{2}I + K_L \right)f$ weakly in $L^2$ as $t\to 0^\pm$.
\end{enumerate}
 \end{lemma}

\section{A ``Calder\'{o}n-Zygmund" Theory for the boundedness of layer potentials:  Proof of Theorem~\ref{P-bdd1}}
\setcounter{equation}{0}
We continue to impose the ``standard assumptions" throughout this section.
We shall work in the upper half-space, the proofs of the analogous bounds for the lower half-space being essentially identical.
Our main goal in this section is to prove Theorem \ref{P-bdd1}.

We begin with some observations concerning the kernels of
the operators $f\mapsto \partial_t \mathcal{S}f(\cdot,t)$ and 
$f\mapsto \nabla_x \mathcal{S}f(\cdot,t),$  
which we denote respectively by 
$$K_t(x,y):=\partial_t E(x,t,y,0)\quad {\rm and } \quad \vec{H}_t(x,y) := \nabla_x E(x,t,y,0).$$
By the DG/N/M estimates \eqref{eq2.DGN} and \eqref{eq2.M} (see \cite[Theorem~3.1]{HK} and \cite[Lemma~2.5]{AAAHK}), for all $t\in\RR$ and $x,y\in\RR^n$ such that $|t|+|x-y|>0$, we have
\begin{align}\label{G-est1}
|E(x,t,y,0)| & \leq  \frac{C}{(|t|+|x-y|)^{n-1}},
\\[4pt]
|\partial_tE(x,t,y,0)| & \leq \frac{C}{(|t|+|x-y|)^{n}}\,,
\label{G-est2}
\end{align}
and, for each integer $m\geq 0$, whenever $2|h|\leq \max(|x-y|,|t|)$,
\begin{align}\begin{split}\label{G-est3}
&|(\partial_t)^m E(x+h,t,y,0)-(\partial_t)^mE(x,t,y,0)|\\
&\qquad+\,|(\partial_t)^mE(x,t,y+h,0)-(\partial_t)^mE(x,t,y,0)| \leq \,C_m\,\frac{|h|^\alpha}{(|t|+|x-y|)^{n+m-1+\alpha}},
\end{split}\end{align}
where $\alpha\in (0,1]$ is the minimum of the De Giorgi/Nash exponents for $L$ and $L^*$ in \eqref{eq2.DGN}. Thus,  $K_t(x,y)$ is 
a standard Calder\'on-Zygmund kernel,
uniformly in $t$, but $\vec{H}_t(x,y)$ is not, 
and for this reason the proof of Theorem \ref{P-bdd1}
will be somewhat delicate.   
On the other hand, the lemma below shows that the kernel
$\vec{H}_t(x,y)$ does satisfy a sort of weak ``1-sided" Calder\'on-Zygmund condition
similar to those considered by Kurtz and Wheeden in  \cite{KW}. In particular, the following lemma from \cite{AAAHK} is at the core of our proof of Theorem~\ref{P-bdd1}.

\begin{lemma}[{\cite[Lemma 2.13, (4.15) and (2.7)]{AAAHK}}]\label{Lemma1}
Suppose that $L$ and $L^*$ satisfy the standard assumptions.
Consider a cube $Q\subset\RR^n$ and fix any points $x,x'\in Q$ and $t,t'\in \RR$ such that $|t-t'|<2\ell(Q)$. For all $(y,s)\in\ree$, set
\[
u(y,s):=E(x,t,y,s)-E(x',t',y,s).
\]
If $\alpha>0$ is the H\"{o}lder exponent in \eqref{G-est3}, then for all integers $k\geq 4$, we have
\[
\sup_{s\in \mathbb{R}}\int_{2^{k+1}Q\setminus 2^kQ}|\nabla u(y,s)|^2\,dy
\leq C2^{-2\alpha k}\Bigl(2^k\ell(Q)\Bigr)^{-n}. 
\]
The analogous bound holds with $E^*$ in place of $E$.
\end{lemma}

We will also need the following lemma from \cite{AAAHK} to deduce \eqref{LP-1} from \eqref{LP-2} for $p\geq2$.

\begin{lemma}[{\cite[Lemma 4.8]{AAAHK}}]\label{l4.8*} Suppose that 
$L$ and $L^*$ satisfy the standard assumptions.
Let $\mathcal{S}_t$ denote the operator $f\mapsto \mathcal{S}_L f(\cdot,t)$.  
Then for $1<p<\infty$, 
\[
\|\widetilde{N}_*\left(\nabla \mathcal{S}_L f\right)\|_{L^p(\RR^n)}\,\leq\, C_p 
\left(1+\sup_{t>0}\|\nabla \mathcal{S}_t\|_{p\to p}\right)\|f\|_{L^p(\RR^n)}\,,
\]
where $\|\cdot\|_{p\to p}$ denotes the operator norm in $L^p$. The analogous bound holds for $L^*$ and in the lower half-space.
\end{lemma}
We are now ready to present the proof of Theorem~\ref{P-bdd1}.

\begin{proof}[Proof of Theorem~\ref{P-bdd1}]  As noted above, we work in $\reu$ and restrict our attention to the layer potentials for $L$, as the proofs in $\mathbb{R}^{n+1}_-$ and for $L^*$ are essentially the same. We first consider estimates \eqref{LP-2}-\eqref{LP-3}, and we separate their proofs into two parts, according to whether $p\leq 2$ or $p>2$.  Afterwards, we prove estimates \eqref{LP-4}-\eqref{LP-6}.

\smallskip

\noindent{\bf Part 1:  estimates \eqref{LP-2}-\eqref{LP-3} in the case} $\p<p\leq 2$.
We set $\mathcal{S} := \mathcal{S}_L^+$ to simplify notation. We separate the proof into the following three parts.

\smallskip

\noindent{\bf Part 1(a):  estimate \eqref{LP-2} in the case} $\p<p\leq 2$. Consider first the case $p\leq 1$.
We claim that if $\frac{n}{n+1}<p\leq 1$ and $a$ is an $H^p$-atom 
in $\RR^n$ with 
\begin{equation}\label{atom-X}
\supp a\subset Q,\quad\int_{\RR^n}a\,dx=0,\quad
\|a\|_{L^2(\RR^n)}\leq\ell(Q)^{n\bigl(\frac{1}{2}-\frac{1}{p}\bigr)},
\end{equation}
then for $\alpha>0$ as in \eqref{G-est3}, and for each integer $k\geq 4$, we have
\begin{equation}\label{eqSt-9}
\sup_{t\geq0}\int_{2^{k+1}Q\setminus 2^kQ}|\nabla{\mathcal{S}}a(x,t)|^2\,dx
\leq C2^{-(2\alpha+n)k}\ell(Q)^{n\bigl(1-\frac{2}{p}\bigr)}\,,
\end{equation}
where $\nabla\mathcal{S}a(\cdot,0)$ is defined on $2^{k+1}Q\setminus 2^kQ$, since $\supp a\subset Q$. Indeed, using the vanishing moment
condition of the atom, Minkowski's 
inequality, and Lemma~\ref{Lemma1} (with the roles of $(x,t)$ and $(y,s)$, or equivalently, the roles
of $L$ and $L^*$, reversed), we obtain
\begin{align*}\begin{split}
\int_{2^{k+1}Q\setminus 2^kQ}|&\nabla{\mathcal{S}}a(x,t)|^2\,dx
\\[4pt]
&=\int_{2^{k+1}Q\setminus 2^kQ}\Bigl|\int_{\RR^n}
\Bigl[\nabla_{x,t}E(x,t,y,0)-\nabla_{x,t}E(x,t,y_Q,0)\Bigr]\,a(y)\,dy
\Bigr|^2\,dx
\\[4pt]
&\leq \left(\int_{\RR^n}|a(y)|\left(\int_{2^{k+1}Q\setminus 2^kQ}
\Bigl|\nabla_{x,t}E(x,t,y,0)-\nabla_{x,t}E(x,t,y_Q,0)\Bigr|^2\,dx\right)^{1/2}
dy\right)^{2}
\\[4pt]
&\leq C2^{-2\alpha k}\Bigl(2^{k}\ell(Q)\Bigr)^{-n}\|a\|_{L^1(\RR^n)}^2
\leq C2^{-(2\alpha+n)k}\ell(Q)^{n\bigl(1-\frac{2}{p}\bigr)},
\end{split}\end{align*}
since
$\|a\|_{L^1(\RR^n)}\leq |Q|^{1/2}\|a\|_{L^2(\RR^n)}$.  This proves ~\eqref{eqSt-9} and thus establishes the claim.

With \eqref{eqSt-9} in hand, we can now prove \eqref{LP-2} by a standard argument. We write
\[
\int_{\RR^n}|\nabla \mathcal{S}a(x,t)|^p\,dx=\int_{16 Q}|\nabla \mathcal{S} a(x,t)|^p\,dx
+\sum_{k=4}^\infty\int_{2^{k+1}Q\setminus 2^kQ}|\nabla \mathcal{S}a(x,t)|^p\,dx,
\]
where $a$ is an $H^p$-atom supported in $Q$ as in \eqref{atom-X}. Applying H\"{o}lder's 
inequality with exponent $2/p$, the $L^2$ estimate for $\nabla \mathcal{S}$ in~\eqref{eq4.prop46}, and estimate \eqref{eqSt-9} for $t>0$, we obtain 
\[
\sup_{t>0}\int_{\RR^n}|\nabla \mathcal{S}a(x,t)|^p\,dx\leq C,
\]
since $\alpha>n(1/p-1)$ in the interval $\p<p\leq1$ with $\p:=n/(n+\alpha)$. This proves (\ref{LP-2}) for $\p<p\leq1$, and so interpolation with \eqref{eq4.prop46} proves (\ref{LP-2}) for $\p<p\leq2$.

\smallskip

\noindent{\bf Part 1(b):  estimate \eqref{LP-1} in the case} $\p<p\leq 2$. We first note that as in Part 1(a), by using \eqref{eq4.9} instead of \eqref{eq4.prop46}, we may reduce matters to showing that for $\p<p\leq 1$ and for each integer $k\geq 10$, we have
\[
\int_{2^{k+1}Q\setminus 2^k Q}
|\widetilde{N}_*(\nabla \mathcal{S} a)|^p\leq C 2^{-(\alpha-n(1/p-1)) kp},
\]
whenever $a$ is an $H^p$-atom supported in $Q$ as in \eqref{atom-X}, since $\alpha >n(1/p-1)$ in the interval $\p<p\leq 1$ with $\p:=n/(n+\alpha)$. In turn, using H\"{o}lder's inequality with exponent $1/p$ when $p<1$,
we need only prove that for each integer $k\geq 10$, we have
\begin{equation}\label{ms1}
\int_{2^{k+1}Q\setminus 2^k Q}
|\widetilde{N}_*(\nabla \mathcal{S} a)|\leq C 2^{-\alpha k}|Q|^{1-1/p}.
\end{equation}

To this end, set $u:={\mathcal{S}}a$, and suppose that $x\in 2^{k+1}Q\setminus 2^k Q$ for some integer $k\geq 10$. We begin with the estimate
$\widetilde{N}_*\leq N_1+N_2$, where 
\begin{align*}
N_1 (\nabla u)(x)&:=\sup_{|x-y|<t<2^{k-3}\ell(Q)}
\left(\fint_{B((y,t),t/4)}|\nabla u|^2\right)^{1/2},\\[4pt]
N_2 (\nabla u)(x)&:=\sup_{|x-y|<t,\,\,t>2^{k-3}\ell(Q)}
\left(\fint_{B((y,t),t/4)} |\nabla u|^2\right)^{1/2}.
\end{align*}
Following \cite{KP}, by Caccioppoli's inequality we have
\begin{align*}\begin{split}
N_1 (\nabla u)(x) &\leq C\sup_{|x-y|< t< 2^{k-3}\ell(Q)}
\left(\fint_{B((y,t),t/2)} \frac{|u-c_B|^2}{t^2}\right)^{1/2}
\\[4pt]
&\leq C \sup_{t<2^{k-3}\ell(Q)}
\left\{\left(\fint_{t/2}^{3t/2}
\fint_{|x-y|<3t/2}\frac{|u(y,s)-u(y,0)|^2}{t^2}\right)^{1/2}\right.
\\[4pt]
& \qquad\qquad\qquad\qquad+\,\,\left.\left(\fint_{|x-y|<3t/2}\frac{|u(y,0)-c_B|^2}{t^2}
\right)^{1/2}\,\right\}
=:I+II,
\end{split}\end{align*}
where the constant $c_B$ is at our disposal, and $u(y,0):=\mathcal{S}a(y,0):=S\!a(y)$. 

By the vanishing moment property of $a$, if $z_Q$ denotes the center of $Q$, then for all $(y,s)$ and $t$ as in $I$, we have
\begin{align*}\begin{split}
\frac{1}{t}|u(y,s)-u(y,0)| &=
\left|\frac{1}{t}\int_0^s\frac{\partial}{\partial\tau}
{\mathcal{S}} a(y,\tau)\,d\tau\right|
\\[4pt]
&\leq \sup_{0<\tau<3t/2}\int_{{\RR}^n}|\partial_\tau E(y,\tau,z,0)
-\partial_\tau E(y,\tau,z_Q,0)||a(z)|\,dz
\\[4pt]
&\leq C\int_{\RR^n}\frac{\ell(Q)^\alpha}{|y-z_Q|^{n+\alpha}}|a(z)|\,dz
\\[4pt]
&\leq C 2^{-\alpha k}(2^k\ell(Q))^{-n}|Q|^{1-1/p},
\end{split}\end{align*}
where in the next-to-last step we have used \eqref{G-est3} with $m=1$, and in the last step we have used that $\|a\|_1 \leq C|Q|^{1-1/p}.$ Thus,
\[
\int_{2^{k+1}Q\setminus 2^kQ} I\,dx\leq 2^{-\alpha k}|Q|^{1-1/p},
\]
as desired. By Sobolev's inequality, for an appropriate 
choice of $c_B$, we have
\begin{align*}\begin{split}
II&\leq C\sup_{0<t<2^{k-3}\ell(Q)}
\left(\fint_{|x-y|<3t/2}|\nabla_{\rm{tan}}u(y,0)|^{2_*}\right)
^{1/2_\ast}
\\[4pt]
&\leq C \left(M\big(|\nabla_{\rm{tan}}u(\cdot,0)|^{2_\ast}
\chi_{2^{k+3}Q\setminus 2^{k-2}Q}\big)(x)\right)^{1/2_*}\,,
\end{split}\end{align*}
where $\nabla_{\rm tan} u(x,0) := \nabla_x u(x,0)$ is the tangential gradient,  
$2_*:= 2n/(n+2)$, and $M$, as usual, denotes the Hardy-Littlewood maximal operator.
Consequently, we have
\begin{align*}\begin{split}
\int_{2^{k+1}Q\setminus 2^kQ}II
&\leq C\left(2^k\ell(Q)\right)^{n/2}
\left(\int_{2^{k+3}Q\setminus 
2^{k-2}Q}|\nabla_{\rm{tan}}u(\cdot,0)|^2\right)^{1/2} \\[4pt]
&\leq C \left(2^k\ell(Q)\right)^{n/2}
2^{-(\alpha+n/2)k}\ell(Q)^{n(1/2-1/p)}
\leq C2^{-\alpha k}|Q|^{1-1/p},
\end{split}\end{align*}
where in the second inequality we used estimate \eqref{eqSt-9} for $t=0$, since $u=\mathcal{S}a$ and ${\supp a \subset Q}$. We have therefore proved
that $N_1$ satisfies \eqref{ms1}. 

It remains to treat $N_2$.  For each $x\in 2^{k+1}Q\setminus 2^kQ$, choose $(y_\ast$, $t_\ast)$ in the cone $\Gamma(x)\subset\RR^{n+1}_+$ so that the
supremum in the definition of $N_2$ is essentially attained, 
i.e., so that
\[
N_2(\nabla u)(x)\leq 2\left(\fint_{B((y_\ast,t_\ast),\,t_\ast/4)}
|\nabla u|^2\right)^{1/2},
\]
with $|x-y_\ast|<t_*$ and $t_\ast\geq 2^{k-3}\ell(Q)$. By Caccioppoli's inequality,
\[
N_2(\nabla u)(x)\leq C\frac{1}{t_\ast}\left(
\fint_{B((y_\ast,t_\ast),\,t_\ast/2)}|u|^2\right)^{1/2}.
\]
Now for $(y,s)\in B((y_\ast,t_\ast),t_\ast/2)$, by \eqref{G-est3} with $m=0$, we have 
\begin{align*}\begin{split}
|u(y,s)|&\leq \int_{\RR^n}|E(y,s,z,0)-E(y,s,z_Q,0||a(z)|\,dz
\\[4pt]
&\leq C\|a\|_{L^1(\RR^n)}\frac{\ell(Q)^\alpha}{s^{n-1+\alpha}}
\leq C\ell(Q)^\alpha t_\ast^{1-n-\alpha}|Q|^{1-1/p}.
\end{split}\end{align*}
Therefore, 
\[
N_2(\nabla u)(x)\leq C\ell(Q)^\alpha t_\ast^{-n-\alpha}|Q|^{1-1/p}
\leq C 2^{-\alpha k}(2^k\ell(Q))^{-n}|Q|^{1-1/p}.
\]
Integrating over $2^{k+1}Q\setminus 2^kQ$, we obtain \eqref{ms1} for $N_2$, hence (\ref{LP-1}) holds for
$\p<p\leq2$.

\smallskip

\noindent{\bf Part 1(c):  estimates \eqref{LP-3a}-\eqref{LP-3} in the case} $\p<p\leq 2$. We note that the case $p=2$ holds by \eqref{eq4.prop46} and Lemma \ref{l4.ntjump} (i) and (ii).  Thus, by interpolation, it is again enough to treat the case $\p<p\leq1$, and in that setting, \eqref{LP-3a}-(\ref{LP-3}) are an immediate consequence of the more general estimates in~\eqref{eq4.31**}-\eqref{eq4.31*} below, which we note for future reference.

\begin{proposition}\label{r4.1} 
Suppose that $L$ and $L^*$ satisfy the standard assumptions, let $\alpha$ denote the minimum of the De Giorgi/Nash exponents for $L$ and $L^*$ in \eqref{eq2.DGN}, and set ${\p:=n/(n+\alpha)}$. Then
\begin{align}\label{eq4.31**}
\|\nabla_x\, S\!f\|_{H^p(\RR^n,\CC^n)} + \sup_{t> 0}\|\nabla_x\, \mathcal{S} f(\cdot, t)\|_{H^p(\RR^n,\CC^n)}&\leq C_p 
\|f\|_{H^p(\RR^n)},\quad\forall\,p\in(\p,1]\,,\\[4pt]
\label{eq4.31*}
\|\partial_\nu\mathcal{S}f(\cdot, 0)\|_{H^p(\RR^n)} + \sup_{t>0}\|\langle A\nabla{\mathcal{S}} f(\cdot, t),e_{n+1}\rangle\|_{H^p(\RR^n)} &\leq C_p 
\|f\|_{H^p(\RR^n)},\quad\forall\,p\in(\p,1]\,,
\end{align}
where $\partial_\nu\mathcal{S}f(\cdot, 0) = ((1/2) I+\widetilde{K})f$ is defined in the variational sense of Lemma~\ref{l4.1}. The analogous results hold for $L^*$ and in the lower half-space.
\end{proposition}

\begin{proof}
It suffices to show that if $a$ is an $H^p(\RR^n)$-atom as in (\ref{atom-X}), and $t>0$, then 
\begin{align*}
&\vec{m}_0:= C\nabla_x S\!a,
&&\vec{m}_t:= C\nabla_x \cS a(\cdot,t),\\[4pt]
&m_0:=C((1/2)I+\widetilde{K})a,
&&m_t:=C\langle A\nabla{\mathcal{S}}a(\cdot, t),e_{n+1}\rangle,
\end{align*}
are all molecules adapted to $Q$, for some harmless constant $C\in(0,\infty)$, 
depending only on the ``standard constants". Recall that, for $n/(n+1)<p\leq 1$, 
an $H^p$-molecule adapted to a cube $Q\subset\RR^n$ 
is a function $m\in L^1(\RR^n)\cap L^2(\RR^n)$ satisfying 
\begin{equation}\label{eqSt-10}
\begin{array}{l}
(i)\quad \int_{\RR^n}m(x)\,dx=0,
\\[8pt]
(ii)\quad \Bigl(\int_{16\,Q}|m(x)|^2\,dx\Bigr)^{1/2}
\leq\ell(Q)^{n\bigl(\frac{1}{2}-\frac{1}{p}\bigr)},
\\[8pt]
(iii)\quad\Bigl(\int_{2^{k+1}Q\setminus 2^kQ}|m(x)|^2\,dx\Bigr)^{1/2}
\leq 2^{-\varepsilon k}\Bigl(2^k\ell(Q)\Bigr)
^{n\bigl(\frac{1}{2}-\frac{1}{p}\bigr)},\qquad\forall\,k\geq 4,
\end{array}
\end{equation}
for some $\varepsilon>0$ (see, e.g., \cite{CW}, \cite{TW}).

Note that for $\vec{m}_t$ and $m_t$, when $t>0$, property (ii) follows from the $L^2$ estimate in \eqref{eq4.prop46}, and (iii) follows from \eqref{eqSt-9} with $\varepsilon:=\alpha-n(1/p-1),$ which is positive for $\p<p\leq 1$ with $\p:=n/(n+\alpha)$. Moreover, these estimates for $\vec{m}_t$ and $m_t$ hold uniformly in $t$, and since $a\in L^2(\RR^n)$, we obtain (ii) and (iii) for $\vec{m}_0$ and $m_0$ by Lemma \ref{l4.ntjump}.

Thus, it remains to show that $\vec{m}_t$ and $m_t$ have mean-value zero for all $t\geq0$. 
This is nearly trivial for $\vec{m}_t$. For any $R>1$, choose $\Phi_R \in C^\infty_0(\RR^{n+1})$, with $0\leq \Phi_R\leq 1$, 
such that
\begin{equation}\label{Phi-R}
\Phi_R\equiv 1\mbox{ on }B(0,R),
\quad \supp \Phi_R\subset B(0,2R),\quad
\|\nabla\Phi_R\|_{L^\infty(\RR^n)}\leq C/R\,,
\end{equation}
and let $\phi_R:=\Phi_R(\cdot,0)$ denote its restriction to $\rn\times\{0\}$.
For $1\leq j\leq n$ and $R > C(\ell(Q)+|y_Q|)$ (where $y_Q$ is the center of $Q$), using that $a$ has mean value
zero, we have
\begin{align*}
\left|\int_{\rn} \partial_{x_j} \cS a (\cdot,t) \,\phi_R \right| &= \left|\int_{\rn} \cS a(\cdot,t)\, \partial_{x_j} \phi_R \right| \\[4pt]
&\lesssim \,\frac1R \int_{R\leq|x|\leq 2R} \int_{Q} |E(x,t,y,0 -E(x,t,y_Q,0)| \, |a(y)| \, dy\, dx\\[4pt]
&\lesssim\, \frac1R \int_{R\leq|x|\leq 2R} \int_{Q} \frac{\ell(Q)^\alpha}{R^{n-1+\alpha}}\, |a(y)|\, dy\\[4pt]
&\lesssim \, \left(\frac{\ell(Q)}{R}\right)^\alpha \|a\|_{L^1(\rn)} \,\lesssim\,
\left(\frac{\ell(Q)}{R}\right)^\alpha \ell(Q)^{n(1-1/p)}\,,
\end{align*}
where we used the DG/N bound \eqref{G-est3} with $m=0$, the Cauchy-Schwarz inequality
and the definition of an atom \eqref{atom-X}.  Letting $R\to \infty$, we obtain $\int_{\rn} \nabla_x \cS a (\cdot,t) = 0$ for all~$t\geq0$.

Next, let us show that $((1/2) I+\widetilde{K})a$ has mean-value zero.
Set $u:={\mathcal{S}}a$ in $\RR^{n+1}_+$, so that matters are reduced to proving that 
\[
\int_{\RR^n} \partial_\nu u(x, 0) \,dx=0,
\]
where $\partial_\nu u(\cdot,0)$ is defined in the variational sense of Lemma~\ref{l4.1}. 
Choose $\Phi_R,\phi_R$ as above, and note that $\partial_\nu u(\cdot,0) \in L^1(\rn)$, by the bounds \eqref{eqSt-10} (ii) and (iii) that we have just established. Then by Lemma~\ref{l4.1}~(i), we have
\begin{align*}\begin{split}
\left|\int_{\RR^n}\partial_\nu u(\cdot,0)\,dx\right|  \,&=\,  
\left|\lim_{R\to\infty}\int_{\RR^n}\partial_\nu u(\cdot,0)\,\phi_R\,dx\right| 
\\[4pt]
 &= \, \left|\lim_{R\to\infty}\int_{\RR^{n+1}_+}\langle A\nabla u,\nabla\Phi_R
\rangle\,dX\right| 
\\[4pt]  &\lesssim \, \overline{\lim_{R\to\infty}}
\left(\int_{X\in{\RR^{n+1}_+}:R<|X|<2R}\,\,|\nabla u|^{q}\,dX\right)
^{1/q}\left(\int_{R<|X|<2R}\,\,|\nabla\Phi_R|^{q'}\,dX\right)
^{1/q'},
\end{split}\end{align*}
where $q:= p(n+1)/n$ and $q' = q/(q-1).$  Since $0<\alpha \leq 1$ and $\p:=n/(n+\alpha)$, we have
 $n/(n+1) <p\leq 1$, hence $1<q\leq (n+1)/n$ and
$n+1\leq q'<\infty.$ Consequently, 
the second factor above is bounded uniformly in $R$ as $R\to \infty$, whilst the first factor converges to zero by Lemma~\ref{L-improve} and the dominated convergence theorem, since we have already proven
\eqref{LP-1} in the case $\p<p\leq 2$. This proves that $\int_{\rn} ((1/2) I+\widetilde{K})a=0$. The proof that $\int_{\rn}\langle A\nabla{\mathcal{S}} a(\cdot, t),e_{n+1}\rangle = 0$ for all $t>0$ follows in the same way, except we use \cite[(4.6)]{AAAHK} instead of Lemma~\ref{l4.1}~(i).
\end{proof}

This concludes the proof of Part 1 of Theorem \ref{P-bdd1}.  At this point, we note for future 
reference the following corollary of \eqref{LP-3} and Proposition~\ref{r4.1}.

\begin{corollary}\label{r4.2}
Suppose that $L$ and $L^*$ satisfy the standard assumptions, and let $\alpha$ denote the minimum of the De Giorgi/Nash exponents for $L$ and $L^*$ in \eqref{eq2.DGN}. Then
\begin{equation}\label{eq4.33*}
\sup_{t> 0}\|\mathcal{D}_L^\pm g(\cdot,t)\|_{\lb(\rn)} \,+\, \|K_L g\|_{\lb(\rn)} \,\leq \,C_\beta\, \|g\|_{\lb(\rn)}\,,
\quad \forall\,\beta\in\left[0,\alpha\right).
\end{equation}
Moreover, $\mathcal{D}_L^\pm 1 $ is constant on $\RR^{n+1}_\pm$, and $K_L1$ is constant on $\RR^n$. The analogous results hold for $L^*$.
\end{corollary}
\begin{proof}
Consider $\reu$ and set $\partial_{\nu^*}\SL_{t,L^*} f:= \partial_{\nu^*} \SL_{L^*}f (\cdot,t)$,
and $\mathcal{D}_{t,L} f:= \mathcal{D}_Lf(\cdot,t)$. We have $\partial_{\nu^*}\SL_{-t,L^*} = \adj(\mathcal{D}_{t,L})$, and by definition $\widetilde{K}_{L^*}=\adj(K_L)$.   Thus, estimates \eqref{LP-3} and \eqref{eq4.31*} imply
\eqref{eq4.33*} by duality. 

The case $\beta=0$ of \eqref{eq4.33*} shows that $\mathcal{D}_L1(\cdot,t)$ and $K_L1$ exist in $BMO(\RR^n)$, for each $t>0$. The moment conditions obtained in the proof of Proposition~\ref{r4.1} show that for any atom $a$ as in (\ref{atom-X}), and for each $t>0$, we have
\[
\langle \mathcal{D}_L 1(\cdot,t), a\rangle = \int_{\rn}    \partial_{\nu^*} \SL_{L^*} a(\cdot,-t) =0\,, \quad \text{ and }\quad
\langle K_L 1, a\rangle = \int_{\rn} \widetilde{K}_{L^*} a =0,\,
\]
where $\langle\cdot,\cdot\rangle$ denotes the dual pairing between $BMO(\rn)$ and $H^1_{\mathrm{at}}(\rn)$.
This shows, since $a$ was an arbitrary atom, that $\mathcal{D}_L 1(\cdot,t)$ and $K_L 1$ are zero in the sense of $BMO(\rn)$, hence $\mathcal{D}_L 1(x,t)$ and $K_L1(x)$ are constant in $x\in\RR^n$, for each fixed $t>0$.

It remains to prove that $\mathcal{D}_L 1(x,t)$ is constant in $t>0$, for each fixed $x\in\RR^n$. To this end, 
let $\phi_R$ denote the boundary trace of a smooth cut-off function $\Phi_R$ as in \eqref{Phi-R}. We observe that by the definition of
$\mathcal{D}_L$ (cf. \eqref{eq2.23}-\eqref{eq2.23*}), and translation invariance in $t$, we have
\begin{align*}
\partial_t \mathcal{D}_L 1(x,t) &=\lim_{R\to \infty} \partial_t \mathcal{D}_L \phi_R(x,t)
=\lim_{R\to \infty}\int_{\mathbb{R}^{n}}\overline{\Big(\partial_t \partial_{\nu^*} E^*
(\cdot,\cdot,x,t)\Big)}(y,0)\,\phi_R(y)\,dy \\[4pt]
&=\lim_{R\to \infty}\int_{\mathbb{R}^{n}}\overline{\partial_s e_{n+1}\cdot A^*(y)\,\Big(\nabla_{y,s} 
E^*(y,s,x,t)\Big)}\big|_{s=0}\,\phi_R(y)dy \\[4pt]
&=\,\lim_{R\to \infty} \sum_{i=1}^n\sum_{j=1}^{n+1}
\int_{\mathbb{R}^{n}}\overline{ A_{i,j}^*(y)\,\Big(\partial_{y_j}
E^*(y,s,x,t)\Big)}\big|_{s=0}\,\partial_{y_i} \phi_R(y)\,dy\,,
\end{align*}
where in the last step we set $y_{n+1}:= s$, and used that $L^* E^* = 0$ away from the pole at $(x,t)$.  The limit above equals 0, since for $R> C|x|$ with $C>1$ sufficiently large, the term at level $R$ is bounded by
\begin{align*}
R^{-1} \int_{C^{-1}R< |x-y|<CR} & \big|\big(\nabla E(x,t,\cdot,\cdot)\big)(y,0)\big|\, dy\\[4pt]
&\lesssim\, R^{-1+n/2} 
\left(\int_{C^{-1}R< |x-y|<CR}\big |\big(\nabla E(x,t,\cdot,\cdot)\big)(y,0)\big|^2\, dy\right)^{1/2}
\lesssim \,R^{-1}\,,
\end{align*}
where in the last step we used the $L^2$ decay for $\nabla E$ from
\cite[Lemma 2.8]{AAAHK}. This shows that $\mathcal{D}_L 1(x,t)$ is constant in $t>0$, for each fixed $x\in\RR^n$, as required.
\end{proof}

\noindent{\bf Part 2:  estimates \eqref{LP-2}-\eqref{LP-3} in the case $2<p<\pp$.} We begin by stating without proof the following variant of Gehring's lemma 
as established by 
Iwaniec~\cite{I}.

\begin{lemma}\label{lemmaIwaniec} Suppose that $g$, $h\in L^p(\RR^n)$,
with $1<p<\infty$, and that for some $C_0>0$ and for all cubes $Q\subset \RR^n$,
\begin{equation}\label{ms13}
\left(\fint_Qg^p\right)^{1/p}
\leq C_0\fint_{4Q} g+ \left(\fint_{4Q}h^p\right)^{1/p}.
\end{equation}
Then there exists $s=s(n,p,C_0)>p$ and $C=C(n,p,C_0)>0$ such that
\begin{equation*}
\int_{\RR^n}g^s\leq C\int_{\RR^n}h^s.
\end{equation*} 
\end{lemma}

\begin{remark}\label{r4.3}
If $1<r<p$, by replacing $g$ with $\tilde{g}:=g^r$,
$h$ with $\tilde{h}:=h^r$, and $p$ with $\tilde{p}:=\frac{p}{r}$, then
the conclusion of the Lemma~\ref{lemmaIwaniec} holds provided (\ref{ms13})
is replaced with
\begin{equation*}
\left(\fint_Qg^p\right)^{1/p}
\leq C_0\left(\fint_{4Q} g^r\right)^{1/r}
+ \left(\fint_{4Q}h^p\right)^{1/p}.
\end{equation*}
In this case $s$ also depends on $r$.
\end{remark}

For the sake of notational convenience, we set
$$\mathcal{S}_t f(x):= \mathcal{S}^\pm_L\,f(x,t), \qquad (x,t)\in \RR^{n+1},$$
so that when $t=0$ we have $\mathcal{S}_0:= S=S_{\!L}$ (cf. \eqref{eq2.24}).
We shall apply the Remark~\ref{r4.3} with 
$g:=\nabla_x \mathcal{S}_{t_0}f$, with $t_0>0$ fixed, $p=2$, and $r=2_{\ast}:=2n/(n+2)$. 
To be precise, we shall prove that for each fixed $t_0>0$, and for every cube $Q\subset\RR^n$, we have
\begin{equation}\label{newSt-3}
\left(\fint_{Q}|\nabla_x \mathcal{S}_{t_0}f|^2\,dx\right)^{1/2}
\lesssim \left(\fint_{4Q}|\nabla_x \mathcal{S}_{t_0}f|^{2_{\ast}}\,dx\right)^{1/2_{\ast}}
+ \left(\fint_{4Q}\Bigl(|f|+N_{**}(\partial_t \mathcal{S}_tf)\Bigr)^2\right)^{1/2}\,,
\end{equation}
for all $f\in L^2\cap L^\infty$, where $N_{**}$ denotes the ``two-sided" nontangential maximal operator
$$N_{**} (u)(x):= \sup_{\{(y,t)\in \RR^{n+1}:\, |x-y|<|t|\}} |u(y,t)|.$$
We claim that
the conclusion of Part 2 of Theorem \ref{P-bdd1}
then follows.   Indeed, for each fixed $t\in \RR$, $\partial_t S_t$ is a Calder\'on-Zygmund operator
with a ``standard kernel'' $K_t(x,y):= \partial_tE(x,t,y,0)$, with Calder\'on-Zygmund constants that are uniform in
$t$ (cf. \eqref{G-est2}-\eqref{G-est3}). Thus, by the $L^2$ bound
\eqref{eq4.prop46}, we have from standard Calder\'on-Zygmund theory and a variant of 
the usual Cotlar inequality for maximal singular integrals that
\begin{equation}\label{eq4.lpnt}\sup_{t>0}\|\partial_t \mathcal{S}_t f\|_{p}
+\|N_{**}(\partial_t \mathcal{S}_t f)\|_p \leq C_p \|f\|_p,\qquad\forall\,p\in(1,\infty).
\end{equation} 
Consequently, if (\ref{newSt-3}) holds for arbitrary $t_0>0$, then there exists $\pp>2$ such that
\[
\sup_{t>0}\|\nabla \mathcal{S}_t f\|_{L^p(\RR^n)}\leq C 
\|f\|_{L^p(\RR^n)},\qquad\forall\,p\in(2,\pp),
\]
by Lemma~\ref{lemmaIwaniec}, Remark~\ref{r4.3} and a density argument. This would prove~\eqref{LP-2}. We could then use Lemma \ref{l4.8*} to obtain \eqref{LP-1}, and \eqref{LP-3a}-\eqref{LP-3} would follow from Lemma \ref{l4.ntjump} (i) and (ii), and another density argument. Thus, it is enough to prove (\ref{newSt-3}).

To this end, we fix a cube $Q\subset\RR^n$ and split
\[
f=f_1+f_2:=f \,1_{4Q}\,+\,f\, 1_{(4Q)^c},\qquad
u=u_1+u_2:=\mathcal{S}_tf_1+\mathcal{S}_tf_2. 
\]
Using \eqref{eq4.prop46}, and the definition of $f_1$, we obtain
\begin{equation}\label{aaahk}
\sup_{t>0}\fint_{Q}|\nabla_x \mathcal{S}_tf_1(x)|^2\,dx
\leq C\fint_{4Q}|f(x)|^2\,dx.
\end{equation}
Thus, to prove \eqref{newSt-3} it suffices to establish 
\begin{equation}\label{newSt-6}
\left(\fint_{Q}|\nabla_x u_2(x,t_0)|^2\,dx\right)^{1/2}
\lesssim  \left(\fint_{4Q}|\nabla_x u(x,t_0)|^{2_{\ast}}
\,dx\right)^{1/2_{\ast}} \, +\,\left(\fint_{4Q}
\Bigl(|f| + N_{**}(\partial_t u)\Bigr)^2\right)^{1/2}.
\end{equation}
To do this, we shall use the following result from \cite{AAAHK}.

\begin{proposition}[{\cite[Proposition 2.1]{AAAHK}}]\label{propslice}  
Let $L$ be as in \eqref{L-def-1}-\eqref{eq1.1*}. If $A$ is $t$-independent, then there exists $C_0>0$, depending only on dimension and ellipticity, such that
for all cubes $Q\subset \RR^n$ and $s\in \RR$ the following holds: if
$Lu=0$ in $4Q\times (s-\ell(Q),s+\ell(Q))$, then  
\[
\frac{1}{|Q|}\int_Q |\nabla u(x,s)|^2 dx \leq C_0 \frac{1}{\ell(Q)^2}
\frac{1}{|Q^{**}|}\iint_{Q^{**}} |u(x,t)|^2 dx dt,
\]
where $Q^{**} := 3Q \times (s-\ell(Q)/2,s+\ell(Q)/2)$.
\end{proposition}

Applying Proposition \ref{propslice} with $s:= t_0$ and $u:=\mathcal{S}_t(f1_{(4Q)^c})$, which is a solution of $Lu=0$ in the infinite strip $4Q\times(-\infty,\infty)$, we obtain
\begin{align*}
\fint_{Q}&|\nabla_x u_2(x,t_0)|^2\,dx \lesssim \frac{1}{\ell(Q)^{2}}
\fint_{t_0-\ell(Q)/2}^{t_0+\ell(Q)/2}
\!\fint_{3Q}|{u_2}(x,t)-c_Q|^2 dx dt
\\[4pt]& \lesssim \,\, \frac{1}{\ell(Q)^{2}}
\fint_{t_0-\ell(Q)/2}^{t_0+\ell(Q)/2}
\!\fint_{3Q}|{u_2}(x,t)-{u_2}(x,t_0)|^2 dx dt + \frac{1}{\ell(Q)^2}
\!\fint_{3Q}|{u_2}(x,t_0)-c_Q|^2 dx
\\[4pt]& \lesssim \,\, \frac{1}{\ell(Q)^{2}}
\fint_{t_0-\ell(Q)/2}^{t_0+\ell(Q)/2}
\!\fint_{3Q}\left|\int^{\max\{t_0,t\}}_{\min\{t_0,t\}} \partial_s {u_2}(x,s)\, ds\right|^2 dx dt + \frac{1}{\ell(Q)^2}
\!\fint_{3Q}|{u_2}(x,t_0)-c_Q|^2 dx \\[4pt]
& \lesssim \,\, \fint_{3Q}\fint_{t_0-l(Q)/2}^{t_0+\ell(Q)/2}|\partial_s{u_2}(x,s)|^2 ds dx
\,+\,\left(\fint_{3Q} |\nabla_x {u_2}(x,t_0)|^{2_*} dx\right)^{2/2_*}
\\[4pt] & \lesssim \,\, \fint_{3Q}\big(N_{**}(\partial_t {u_2})(x)\big)^2dx +
\left(\fint_{3Q} |\nabla_x {u_2}(x,t_0)|^{2_*} dx\right)^{2/2_*}
\\[4pt] & \lesssim \,\,  \fint_{3Q}\big(N_{**}(\partial_t {u})(x)\big)^2dx +
\left(\fint_{3Q} |\nabla_x {u}(x,t_0)|^{2_*} dx\right)^{2/2_*} +
 \fint_{4Q} |f(x)|^2dx,
\end{align*}
where in the third last estimate we have made an appropriate choice of $c_Q$ 
in order to use Sobolev's inequality, and in the last estimate we wrote $u_2=u-u_1$, and then used \eqref{eq4.lpnt} with $p=2$ to control $N_{**}(\partial_t u_1)$, and 
\eqref{aaahk} to control $\nabla_x u_1$. Estimate \eqref{newSt-6} follows, and so the proofs of estimates \eqref{LP-2}-\eqref{LP-3} are now complete.

\smallskip

\noindent{\bf Part 3:  proof of estimate \eqref{LP-4}}\footnote{We are indebted to S. Mayboroda for suggesting this proof, which simplifies our original argument.}.  
We shall actually prove a more general result.  It will be convenient to use the following notation. For $f\in L^2(\rn,\CC^{n+1})$, set
\begin{equation}\label{eq7.8a}
\big(\SL^\pm\nabla\big) f(x,t) := \int_{\rn} \Big(\nabla E(x,t,\cdot,\cdot)\Big)(y,0)\, \cdot f(y)\,dy\,,\qquad (x,t)\in\RR^{n+1}_\pm\,,
\end{equation} 
where $\Big(\nabla E(x,t,\cdot,\cdot)\Big)(y,0):=\Big(\nabla_{y,s} E(x,t,y,s)\Big)\big|_{s=0}\,$. Our goal is to prove that
\begin{equation}\label{LP-4a}
\big\|N_*\big((\SL \nabla) f\big)\big\|_{L^p(\RR^n)} \, \leq \,C_p\|f\|_{L^p(\RR^n,\CC^{n+1})}, 
\qquad\forall\,p\in\left(\frac{\pp}{\pp-1},\infty\right)\,,
\end{equation} 
which clearly implies \eqref{LP-4}, by the definition of $\mathcal{D}$. It is enough to prove \eqref{LP-4a} for all $f\in L^p\cap L^2$.  Moreover,  it is enough to work with $\widetilde{N}_*$ rather than $N_*$, since for solutions, the former controls the latter pointwise, 
for appropriate choices of aperture, by the Moser estimate \eqref{eq2.M}.

In view of Lemma~\ref{l4.ntjump}, we extend definition \eqref{eq7.8a} to the boundary of $\RR^{n+1}_\pm$ by setting
\[
\big(\SL^\pm\nabla\big) f(\cdot,0) := \adj\left(\mp\frac{1}{2A_{n+1,n+1}}e_{n+1} + {\bf T}\right)f\,.
\]
By duality, since \eqref{LP-2} holds for $L^*$, as well as for $L$, and in both half-spaces, we have
\begin{equation}\label{eq4.65}
\sup_{\pm t\geq0}\|(\SL^\pm \nabla)  f(\cdot,t)\|_p \leq C_p \|f\|_p\,,\qquad\forall\, p\in\left(\frac{\pp}{\pp-1},\infty\right)\,,\end{equation}
for all $f\in L^p\cap L^2$, where we used Lemma \ref{l4.ntjump} (ii) to obtain the bound for $t=0$.

We now fix $p>\pp/(\pp-1)$ and choose $r$ so that $\pp/(\pp-1)<r<p$. Let $z\in \RR^n$ and $(x,t)\in \Gamma(z)$. For each integer $k\geq0$, set $\Delta_k:= \{y\in \RR^n:  |y-z|< 2^{k+2}t\}\,$ and write
$$f=f\,1_{\Delta_0}+\sum_{k=1}^\infty f\,1_{\Delta_k\setminus\Delta_{k-1}}
=: f_0+\sum_{k=1}^\infty f_k=: f_0 +f^\star\,.$$ 
Likewise, set $u:= (\SL \nabla) f,\,u_k:= (\SL \nabla) f_k\,$, and $ u^\star:= (\SL \nabla) f^\star$.
Also, let $B_{x,t}:=B((x,t),t/4)$, $\tB_{x,t}:= B((x,t),t/2)$
and $B^k_{x,t}:=B((x,t),2^kt)$ 
denote the Euclidean balls in $\ree$ centered at $(x,t)$ of radius $t/4$, $t/2$, and $2^kt$, respectively.
Using the Moser estimate \eqref{eq2.Mr}, we have
\begin{align*}\left(\fiint_{B_{x,t}}|u|^2\right)^{1/2} &\lesssim \left(\fiint_{\tB_{x,t}}|u|^r\right)^{1/r} 
\leq \left(\fiint_{\tB_{x,t}}|u_0|^r\right)^{1/r}
+\left(\fiint_{\tB_{x,t}}|u^\star|^r\right)^{1/r}
\\[4pt] 
&\lesssim\,
\left(\fiint_{\tB_{x,t}}|u_0|^r\right)^{1/r}
+\,\left(\fiint_{\tB_{x,t}}|u^\star(y,s)-u^\star(y,0)|^r\,dyds\right)^{1/r}
\\[4pt]&\qquad+
\left(\fint_{\Delta_0}|u(\cdot,0)|^r\right)^{1/r}+\left(\fint_{\Delta_0}|u_0(\cdot,0)|^r\right)^{1/r}
\\[4pt]
&=:I+II+III+IV\,.
\end{align*}
By \eqref{eq4.65}, we have
\[
I +IV\lesssim \left(\fint_{\Delta_0} |f|^r\right)^{1/r}\,.
\]
To estimate $II$, note that for all $(y,s)\in B^k_{x,t}$, including when $s=0$, we must have 
\[
u_k(y,s) := (\SL \nabla) f_k(y,s) = \int_{\rn} \Big(\nabla E(y,s,\cdot,\cdot)\Big)(z,0)\, \cdot f_k(z)\,dz\,,
\]
since $\supp f_k \subset \Delta_k\setminus\Delta_{k-1}$ does not intersect $B^k_{x,t}$. Thus, $L u_k = 0$ in $B^k_{x,t}$, and so by the De~Giorgi/Nash estimate~\eqref{eq2.DGN}, followed by \eqref{eq2.Mr} and \eqref{eq4.65}, we obtain
$$II\leq \sum_{k=1}^\infty \sup_{B_{x,t}} |u_k(y,s)-u_k(y,0)|
\lesssim \sum_{k=1}^\infty2^{-\alpha k}
\left(\fiint_{B^k_{x,t}}|u_k|^r\,\right)^{1/r}
\lesssim \sum_{k=1}^\infty2^{-\alpha k}\left(\fint_{\Delta_k}|f|^r\,\right)^{1/r}\,.$$
We also have $III \leq (M(|u(\cdot,0)|^r)(z))^{1/r}$, where $M$ denotes the Hardy-Littlewood maximal operator. Altogether, for all $z\in\RR^n$, after taking the supremum over $(x,t)\in\Gamma(z)$,
we obtain
\begin{equation}\label{eq4.pwnm} 
N_*((\SL \nabla) f)(z)\lesssim \widetilde{N}_*((\SL \nabla) f)(z)\lesssim \big(M\big(|(\SL \nabla) f(\cdot,0)|^r\big)(z)\big)^{1/r}
+\big(M\big(|f|^r\big)(z)\big)^{1/r}\,.\end{equation}
Estimate \eqref{LP-4a}, and thus also \eqref{LP-4}, then follow readily from
\eqref{eq4.pwnm} and \eqref{eq4.65}.

\smallskip

\noindent{\bf Part 4:  proof of estimate \eqref{LP-5}}.
We first recall the following square function estimate, whose proof is given in \cite[Section 3]{HMaMo}:
\begin{equation*}
\iint_{\mathbb{R}_+^{n+1}}\left|t\,\nabla\left(\SL\nabla\right)f(x,t)\right|^{2}\,\frac{dxdt}{t}
 \lesssim\,\Vert
f\Vert_{L^2(\rn)}^{2}\,.
\end{equation*}
By the definition of $\D$, this implies in particular that
\begin{equation}
\iint_{\mathbb{R}_+^{n+1}}\left|t\,\nabla\D f(x,t)\right|^{2}\,\frac{dxdt}{t}
 \lesssim\,\Vert
f\Vert_{L^2(\rn)}^{2}\,.\label{eqA1}
\end{equation}

We now proceed to prove \eqref{LP-5}.    
The proof follows a classical argument of \cite{FS}. Fix a cube $Q$, set $Q_0:=32Q$ and observe that since
$\nabla\D 1=0$ by Corollary~\ref{r4.2}, 
we may assume without loss of generality that $f_{Q_0}:=\fint_{Q_0} f=0$.
We split $f=f_0+\sum_{k=4}^\infty f_k$, where $f_0:= f1_{Q_0}$, and 
$f_k:= f1_{2^{k+2}Q\setminus 2^{k+1} Q}$, $k\geq 4$.

By \eqref{eqA1}, we have
\begin{equation*}
\int_0^{\ell(Q)}\!\!\!\int_Q \left|t\,\nabla\D f_0(x,t)\right|^{2}\,\frac{dxdt}{t}
 \lesssim\, \int_{Q_0} |f|^2\,= \int_{Q_0} |f-f_{Q_0}|^2\,\lesssim \,
|Q|\, \|f\|^2_{BMO(\rn)}\,.
\end{equation*}
Now suppose that $k\geq 4$.  Since $u_k:=\D f_k$ solves $Lu_k=0$ in
$4Q\times(-4\ell(Q),4\ell(Q))$, by Caccioppoli's inequality, we have
$$\int_0^{\ell(Q)}\!\!\!\int_Q \left|t\,\nabla\D f_k(x,t)\right|^{2}\,\frac{dxdt}{t}
 \lesssim\,\fint_{-\ell(Q)}^{2\ell(Q)}\!\!\int_{2Q} \left|\D f_k(x,t)-c_{k,Q}\right|^{2}\,dxdt\,,$$
 where the constant $c_{k,Q}$ is at our disposal.  We now choose $c_{k,Q}:= 
 \D f_k(x_Q,t_Q)$, where $x_Q$ denotes the center of $Q$, and $t_Q>0$ is chosen
such that $|t-t_Q|<2\ell(Q)$ for all $t\in (-\ell(Q),2\ell(Q))$.
We observe that, by the Cauchy-Schwarz inequality,
\begin{align}\begin{split}\label{eq4.44}
\big|\D f_k&(x,t)-c_{k,Q}\big|^{2}\\[4pt]
&\leq\, \int_{2^{k+2}Q\setminus2^{k+1}Q}\left|\nabla_{y,s}\big(E(x,t,y,s) - 
E(x_Q,t_Q,y,s)\big)|_{s=0}\right|^2dy\,\int_{2^{k+2}Q\setminus2^{k+1}Q}|f|^2\\[4pt]
&\lesssim \, 2^{-2k\alpha}\fint_{2^{k+2}Q}|f|^2\,=\,
2^{-2k\alpha}\fint_{2^{k+2}Q}|f-f_{Q_0}|^2\lesssim\, 
k^2 \,2^{-2k\alpha}\, \|f\|^2_{BMO}\,,
\end{split}\end{align}
where we have used Lemma \ref{Lemma1}, and then the telescoping argument of \cite{FS}, 
in the last two inequalities.  
Summing in $k$, we obtain \eqref{LP-5}.

\smallskip

\noindent{\bf Part 5:  proof of estimate \eqref{LP-6}}. It suffices to prove that $\|\mathcal{D} f\|_{\dot{C}^\beta({\reu})} \leq  C_\beta \|f\|_{\dot{C}^\beta(\rn)}\,$, since then $\mathcal{D} f$ has an extension in $\dot{C}^\beta(\overline{\reu})$. Moreover, by Corollary~\ref{cor4.47} below, this extension must then satisfy $\mathcal{D}_L^{\pm }g (\cdot,0) = (\mp\frac{1}{2}I + K_L)g$.

For $a> 0$, let $I:= Q\times [a, a+\ell(Q)]\,$ denote any $(n+1)$-dimensional cube contained in ${\reu}$, 
where as usual,  $Q$ is a cube in $\rn$. 
By the well-known criterion of N. Meyers
\cite{Me}, it is enough to show that for every such $I$, there is a constant $c_I$ such that
\begin{equation}\label{eq4.45}
\frac1{|I|}\iint_I\, \frac{|\D f(x,t)- c_I|}{\ell(I)^{\,\beta}}
\,dxdt\, \leq \,C_\beta\, 
\|f\|_{\dot{C}^\beta(\rn)}\,,
\end{equation}
for some uniform constant $C_\beta$ depending only on $\beta$
and the standard constants. To do this, we set $c_I := \fint_Q \D f(\cdot, a +\ell(Q))$, and $Q_0:=32Q$. For this choice of $c_I$, we may suppose without loss of generality that
$f_{Q_0}:=\fint_{Q_0} f =0$, since $\nabla\mathcal{D} 1=0$ by Corollary~\ref{r4.2}.
We then make the same splitting $f = f_0 + \sum_{k=4}^\infty f_k$ as in Part 4 above,
which in turn induces a corresponding splitting $c_I=\sum c_{k,I}$,
with $c_{k,I} := \fint_Q \D f_k(\cdot,a +\ell(Q))$.
By \eqref{LP-4}, we have
\[
\sup_{t> 0}\|\D f(\cdot,t)\|_{L^2(\rn)} \leq C \|f\|_{L^2(\rn)}\,.
\]
Consequently, since $\ell(I)=\ell(Q)$,
\begin{align*}
\frac1{|I|}\iint_I\,&\frac{|\D f_0(x,t)- c_{0,I}|}{\ell(I)^{\,\beta}}
\,dxdt\, \leq\,
\frac1{\ell(Q)^{\,\beta}}\, \sup_{t>0} \fint_Q |\D f_0(x,t)| \, dx
\\[4pt] &\leq \, C\ell(Q)^{-\beta} \left(\fint_{Q_0}|f|^2\right)^{1/2}
\,=\, C\ell(Q)^{-\beta} \left(\fint_{Q_0}|f-f_{Q_0}|^2\right)^{1/2}\,\leq
 \,C\, 
\|f\|_{\dot{C}^\beta(\rn)}\,.
\end{align*}
For $k\geq 4$, we then observe that, exactly as in \eqref{eq4.44},
we have
\[
\ell(I)^{-\beta}\left|\D f_k(x,t)-c_{k,I}\right| 
\,\lesssim \, 2^{-k\alpha}\ell(Q)^{-\beta}\left(\fint_{2^{k+2}Q\setminus2^{k+1}Q}|f-f_{Q_0}|^2\right)^{1/2}\,\lesssim\, 
k \,2^{-k(\alpha-\beta)}\, \|f\|_{\dot{C}^\beta(\rn)}\,,
\]
where in the last step we have again used the telescoping argument of \cite{FS}. Summing in~$k$, we obtain \eqref{LP-6}. This completes the proof of Theorem \ref{P-bdd1}.
\end{proof}

We conclude this section with the following immediate corollary of Theorem \ref{P-bdd1}. We will obtain more refined versions of some of these convergence results in Section~\ref{s6}.

\begin{corollary}\label{cor4.47}
Suppose that $L$ and $L^*$ satisfy the standard assumptions, let $\alpha$ denote the minimum of the De Giorgi/Nash exponents for $L$ and $L^*$ in \eqref{eq2.DGN}, set $\p:=n/(n+\alpha)$ and let $\pp>2$ be as in Theorem \ref{P-bdd1}.

\noindent If $1<p<\pp$ and $\pp/(\pp-1)<q<\infty$, then for all $f\in L^p(\mathbb{R}^n)$ and $g\in L^q(\rn)$, one has
\begin{enumerate}
\item[(i)] $-\langle A\nabla{\mathcal{S}}_L^\pm f(\cdot, t),e_{n+1}\rangle \to
\left(\pm \frac{1}{2}I + \widetilde{K}_L\right)f$ weakly in $L^p$ as $t\to 0^\pm$.
\item[(ii)]  $\nabla \mathcal{S}^\pm_L f(\cdot,t) 
\to \left(\mp\frac{1}{2A_{n+1,n+1}}e_{n+1} + {\bf T}_L\right)f$ weakly in $L^p$ as $t\to 0^\pm$.
\item[(iii)]$\mathcal{D}_L^{\pm }g (\cdot,t)
\to \left(\mp\frac{1}{2}I + K_L \right)g$ weakly in $L^q$ as $t\to 0^\pm$.
\end{enumerate}
If $\p<p\leq 1$ and $0\leq \beta<\alpha$, then for all $f\in H^p(\rn)$ and $g\in \lb(\rn)$, one has
\begin{enumerate}
\item[(iv)] $-\langle A\nabla{\mathcal{S}}_L^\pm f(\cdot, t),e_{n+1}\rangle \to
\left(\pm \frac{1}{2}I + \widetilde{K}_L\right)f$ in the sense of tempered distributions as $t\to 0^\pm$.
\item[(v)]  $\nabla_x \mathcal{S}^\pm_L f(\cdot,t) 
\to \nabla_x S_{\!L}f$ in the sense of tempered distributions as $t\to 0^\pm$.
\item[(vi)] $\mathcal{D}_L^{\pm }g (\cdot,t)\to \left(\mp\frac{1}{2}I + K_L \right)g$ in the weak* topology on $\lb(\rn)$, $0\leq\beta<\alpha$, as $t\to 0^\pm$. Moreover, if $0<\beta<\alpha$, then $\mathcal{D}_L^{\pm }g (\cdot,0) = \left(\mp\frac{1}{2}I + K_L \right)g$ in the sense of $\dot{C}^\beta(\rn)$.
\end{enumerate}
\noindent If $\p<p<\pp$, then for all $f\in H^p(\rn)$, one has
\begin{enumerate}
\item[(vii)]  $\mathcal{S}^\pm_L f(\cdot,t) \to S_{\!L}f$ in the sense of tempered distributions as $t\to 0^\pm$.
\end{enumerate}
The analogous results hold for $L^*$.
\end{corollary}

\begin{proof}
Items (i)-(v) of the corollary follow immediately from Theorem \ref{P-bdd1}, Proposition~\ref{r4.1}, Lemma \ref{l4.ntjump},
and the fact that $L^2\cap H^p$ is dense in $H^p$ for $0<p<\infty$. We omit the details. Item (vii) is ``elementary" in the case $p>1$, by \eqref{G-est1}-\eqref{G-est3}. The case $\p<p\leq 1$ follows readily from the case $p=2$, the density of $L^2\cap H^p$ in $H^p$, estimate \eqref{LP-3a}, Proposition~\ref{r4.1}, and the Sobolev embedding (see, e.g., \cite[III.5.21]{St2})$$\|h\|_{L^q(\rn)} \lesssim \|\nabla_x h\|_{H^p(\rn)}\,,\qquad \frac1q=\frac1p -\frac1n,$$ since $p>n/(n+1)$ implies $q>1$. Again we omit the routine details.
 
We prove item (vi) as follows, treating only layer potentials for $L$ in $\reu$, 
 as the proofs for $L^*$ and in $\ree_-$ are the same. 
We recall that $\lb(\rn) = (H^p_{\mathrm{at}}(\rn))^*$, with $p=n/(n+\beta)$
(so that, in particular, $n/(n+1) <p\leq 1$). 
It is therefore enough to prove that
\begin{equation}\label{7.dualitylimit}
\left\langle a, \mathcal{D}_Lg(\cdot,t) \right\rangle \xrightarrow{t\to 0^+} 
\left\langle a,\Big(( - 1/2) I + K_L\Big)g
\right\rangle,
\end{equation} 
where $a$ is an $H^p(\rn)$-atom supported in a cube $Q\subset\rn$ as in \eqref{atom-X}, and $\langle\cdot,\cdot\rangle$ denotes the dual pairing between $H^p_{\mathrm{at}}(\rn)$ and $\Lambda^\beta(\rn)$. Using Corollary~\ref{r4.2} and dualizing, we have  
\begin{align*}\begin{split}
\left\langle a, \mathcal{D}_Lg(\cdot,t) \right\rangle 
&=\, \left\langle a, \mathcal{D}_L(g-g_Q)(\cdot,t) \right\rangle
=\, \left \langle \partial_{\nu^*} \SL_{L^*}a (\cdot,-t), g-g_Q \right\rangle \\[4pt]
&=\, \left \langle \partial_{\nu^*} \SL_{L^*}a (\cdot,-t), (g-g_Q) 1_{\lambda Q} \right\rangle \,+\,
\left \langle \partial_{\nu^*} \SL_{L^*}a (\cdot,-t), (g-g_Q) 1_{(\lambda Q)^c} \right\rangle \\[4pt]
&=: I_t(\lambda)+II_t(\lambda)\,,
\end{split}\end{align*}
where $g_Q:= \fint_Q g$, and $\lambda>0$ is at our disposal.
Since $a \in L^2$ and $(g-g_Q) 1_{\lambda Q}\in L^2$, by Lemma \ref{l4.ntjump}, we have
\[
I_t(\lambda) \xrightarrow{t \to 0^+}\left \langle \left((-1/2)I + \widetilde{K}_{L^*}\right)a, (g-g_Q) 
1_{\lambda Q} \right\rangle 
=\left \langle  a,\Big(( - 1/2) I + K_{L}\Big)(g-g_Q) 1_{\lambda Q} \right\rangle\,. 
\]
Setting $R_j:=2^{j+1} \lambda Q\setminus 2^j \lambda Q$,
we have
\begin{align*}
|II_t(\lambda)| 
\, &\leq\, \sum\limits_{j=0}^\infty \| \partial_{\nu^*} \SL_{L^*}a (\cdot,-t)\|_{L^2(R_j)}
\,\|g-g_{Q}\|_{L^2(R_j)}
\\[4pt] &\lesssim \sum_j (2^j\lambda)^{-\alpha-\frac{n}{2}}|Q|^{\frac12-\frac1p} 
\,\|g-g_{Q}\|_{L^2(R_j)}
\,\lesssim\, \sum\limits_j (2^j\lambda)^{\,(\beta-\alpha)/2} \|g\|_{\Lambda^\beta}  \approx \lambda^{\,(\beta-\alpha)/2}\|g\|_{\Lambda^\beta},
\end{align*}
where in the second and third inequalities, we used 
\eqref{eqSt-9} and then a telescoping argument.  We observe that the bound \eqref{eqSt-9} is uniform in $t$,
so that 
$$\lim_{\lambda\to \infty} II_t(\lambda) = 0\,,\quad {\rm uniformly \,\,in}\,\, t\,.$$
Similarly, by \eqref{eqSt-10}~(iii) with $m:=((-1/2)I + \widetilde{K}_{L^*})a$, we have
\begin{align*}
\Big\langle  a,\Big(( - 1/2) I + K_{L}\Big)(g-&g_Q) \textbf{1}_{(\lambda Q)^c} \Big\rangle\\[4pt] 
&=\,\Big\langle \left((-1/2)I + \widetilde{K}_{L^*}\right)a, (g-g_Q) \textbf{1}_{(\lambda Q)^c} \Big\rangle 
\lesssim \lambda^{\,(\beta-\alpha)/2}\|g\|_{\Lambda^\beta}\, \to\, 0\,,
\end{align*}
as $\lambda \to \infty$. Using Corollary~\ref{r4.2} again, we then obtain \eqref{7.dualitylimit}.
\end{proof}

\section{Solvability via the method of layer potentials: Proof of Theorem \ref{th2}}
\setcounter{equation}{0}
The case $p=2$ of Theorem~\ref{th2} was proved in \cite{AAAHK} via the method of layer potentials. We shall now use Theorem~\ref{P-bdd1} and perturbation techniques to extend that result to the full range of indices stated in Theorem~\ref{th2}.

\begin{proof}[Proof of Theorem \ref{th2}]
Let $L:=-\dv A\nabla$ and $L_0:=-\dv A_0\nabla$ be as in \eqref{L-def-1} and \eqref{eq1.1*}, with $A$ and $A_0$ both $t$-independent, and suppose that $A_0$ is real symmetric. Let $\eps_0>0$ and suppose that $\|A-A_0\|_\infty <\eps_0$. We suppose henceforth that $\eps_0>0$ is small enough (but not yet fixed) so that, by Remark~\ref{firstrem}, every operator $L_\sigma:= (1-\sigma)L_0 +\sigma L,\, 0\leq\sigma\leq 1$, along with its Hermitian adjoint, satisfies the standard assumptions, with uniform control of the ``standard constants". For the remainder of this proof, we will let $\epp$ denote an arbitrary small positive number, not necessarily the same at each occurrence, but ultimately depending only on the standard constants and the perturbation radius $\eps_0$.

Let us begin with the Neumann and Regularity problems. As mentioned in the introduction,  for real symmetric coefficients (the case $A=A_0$), solvability of $(N)_p$ and $(R)_p$ was obtained in \cite{KP} in the range $1\leq p <2+\epp$.
Moreover, although not stated explicitly in \cite{KP}, the methods of that paper provide the analogous Hardy space results in the range $1-\epp<p<1$, but we shall not use this fact here. We begin with two key observations.

Our first observation is that, by Theorem \ref{P-bdd1} and analytic perturbation theory,
\begin{align}\begin{split}\label{eq5.1a}
\|\big(\widetilde{K}_{L} - \widetilde{K}_{L_0}\big) f\|_{H^p(\rn)} \,+\,
\|\nabla_x&\left(S_{\!L} - S_{\!L_0}\right) f\|_{H^p(\rn)}\\[4pt]
&\leq\, C_p\, \|A-A_0\|_{L^\infty(\rn)}\,\|f\|_{H^p(\rn)}\, ,\quad 1-\epp < p<2+\epp\,.
\end{split}\end{align}
We may verify~\eqref{eq5.1a} by following the arguments in \cite[Section 9]{HKMP2}.

Our second observation is that we have the pair of estimates
\begin{equation}\label{eq5.2a}
\|f\|_{H^p(\rn)}\,\leq \,C_p\,
\|\big((\pm1/2)I + \widetilde{K}_{L_0}\big) f\|_{H^p(\rn)} 
\, ,\quad 1\leq p<2+\epp\,,
\end{equation}
and 
\begin{equation}\label{eq5.3a}
\|f\|_{H^p(\rn)}\,\leq \,C_p\,
\|\nabla_xS_{\!L_0} f\|_{H^p(\rn)}\, ,\quad 1\leq p<2+\epp\,.
\end{equation}
We verify  \eqref{eq5.2a} and \eqref{eq5.3a} by using Verchota's argument 
in \cite{V} as follows.  First, it is enough to establish these estimates
for $f\in L^2(\rn)\cap H^p(\rn)$, 
which is dense in $H^p(\rn)$.
Then, by the triangle inequality, we have
\begin{align}\begin{split}\label{eq5.4a}
C_p \|f\|_{H^p(\rn)}\,&\leq\,
\|\big((1/2)I + \widetilde{K}_{L_0}\big) f\|_{H^p(\rn)} \,+ \,\|\big((1/2)I - \widetilde{K}_{L_0}\big) f\|_{H^p(\rn)} 
\\[4pt] &= \|\partial_\nu u_0^+\|_{H^p(\rn)} + \|\partial_\nu u_0^-\|_{H^p(\rn)}\,,
\end{split}\end{align}
where $u_0^\pm :=\mathcal{S}_{L_0}^\pm f$, and we have used the jump relation formula in
Lemma~\ref{l4.ntjump}~(i). Moreover, by the solvability of $(N)_p$ and $(R)_p$ in \cite{KP}, 
which we apply in both the upper and lower half-spaces,
and the fact that the tangential gradient of the single layer potential does not jump across the boundary, we have that
\begin{align*}\begin{split}
\|\partial_\nu u_0^+\|_{H^p(\rn)}
&\approx \|\nabla_x u^+_0(\cdot,0)\|_{H^p(\rn)}
\\[4pt]
&=\,\|\nabla_x u^-_0(\cdot,0)\|_{H^p(\rn)}
\approx \|\partial_\nu u_0^-\|_{H^p(\rn)}\,,\quad  1\leq p<2+\epp\,,
\end{split}\end{align*}
where the implicit constants depend only on $p$, $n$ and ellipticity.
Combining the latter estimate with \eqref{eq5.4a}, we obtain
\eqref{eq5.2a} and \eqref{eq5.3a}.

With \eqref{eq5.2a} and \eqref{eq5.3a} in hand, we obtain
invertibility of the mappings
$$(\pm1/2)I + \widetilde{K}_{L_0}: H^p(\rn)\to H^p(\rn)\,,\,\,
{\rm and} \,\,
 S_{\!L_0}: H^p(\rn) \to \dot{H}^p_1(\rn)\,,$$
by a method of continuity argument which connects
$L_0$ to the Laplacian $-\Delta$ via the path $\tau\to L_\tau := (1-\tau)(-\Delta) +\tau L_0,\, 0\leq\tau\leq1$.
Indeed, the ``standard constants" are uniform for every 
$L_\tau$ in the family, so we have the analogue of \eqref{eq5.1a}, with $L$ and $L_0$ 
replaced by 
$L_{\tau_1}$ and $L_{\tau_2}$, 
for any $\tau_1,\tau_2 \in [0,1]$. We omit the details.

We now fix $\eps_0>0$ small enough, depending on the constants in \eqref{eq5.1a}-\eqref{eq5.3a}, so that \eqref{eq5.2a} and \eqref{eq5.3a} hold with $L$ in place of $L_0$.
Consequently, by another method of continuity argument, in which we now
connect $L$ to $L_0$, via the path $\sigma\mapsto L_\sigma:= (1-\sigma)L_0 +\sigma L,\, 0\leq\sigma\leq 1$,
and use that \eqref{eq5.1a} holds not only for $L$ and
$L_0$, but also uniformly for any intermediate pair $L_{\sigma_i}$ and $L_{\sigma_2}$,
we obtain
invertibility of the mappings
\begin{equation}\label{eq5.5a}
(\pm1/2)I + \widetilde{K}_{L}: H^p(\rn)\to H^p(\rn)\,,\,\,
{\rm and} \,\,
 S_{\!L}: H^p(\rn) \to \dot{H}^p_1(\rn)\,,
 \end{equation}
 initially in the range $1\leq p<2+\epp$.
 Again we omit the routine details.  Moreover, since \eqref{LP-3a}-\eqref{LP-3} hold in the range
 $n/(n+\alpha)<p<\pp$, we apply the extension of Sneiberg's Theorem obtained in \cite{KM}
 to deduce that the operators in \eqref{eq5.5a} are invertible in the range 
 $1-\epp<p<2+\epp$. We then apply the extension of the open mapping theorem obtained in \cite[Chapter~6]{MMMM}, which holds on quasi-Banach spaces, to deduce that the inverse operators are bounded.

At this point, we may construct solutions of $(N)_p$ and $(R)_p$ as follows.  
Given Neumann
data $g\in H^p(\rn)$, or Dirichlet
data $f\in \dot{H}^{1,p}(\rn)$, with  $1-\epp<p<2+\epp$, we set
$$u_{N}:= \mathcal{S}_L \left((1/2)I +\widetilde{K}_L\right)^{-1} g\,,
\qquad u_{R}:= \SL_L(S_{\!L}^{-1} f)$$
and observe that $u_N$ then solves $(N)_p$, and $u_R$ solves $(R)_p$,
by \eqref{LP-1}, the invertibility of $(1/2)I +\widetilde{K}_L$ and of $S_{\!L}$ (respectively),
and Corollary \ref{cor4.47} (which guarantees weak convergence to the data;  we
defer momentarily the matter of non-tangential convergence). 

Next, we consider the Dirichlet problem.  
Since the previous analysis also applies to $L^*$,
we dualize our estimates for $(\pm1/2)I +\widetilde{K}_{L^*}$ to obtain that
$(\pm1/2)I+K_{L}$ is bounded and invertible from $L^q(\rn)$ to $L^q(\rn)$, 
$2-\epp <q<\infty$, and from $\lb(\rn)$ to $\lb(\rn)$, $0\leq \beta<\epp$.
Given Dirichlet data $f$ in $L^q(\rn)$ or $\lb(\rn)$, in the stated ranges,
we then construct the solution to the Dirichlet problem by setting
$$u:= \mathcal{D}_L \Big((-1/2)I +K_{L}\Big)^{-1} f\,,$$
which solves $(D)_q$ (at least in the sense of weak convergence to the data) or $(D)_{\lb}$, by virtue of \eqref{LP-4}-\eqref{LP-6} and Corollary \ref{cor4.47}. 

We note that, at present, our solutions to $(N)_p$, $(R)_p$, and $(D)_q$
assume their boundary data in the weak sense of Corollary \ref{cor4.47}. In the next section, however, we establish some results of Fatou type (see Lemmata \ref{lemma6.1}, \ref{lemma6.2a} and \ref{lemma6.2}), which allow us to immediately deduce the stronger non-tangential and norm convergence results required here.

It remains to prove that our solutions to $(N)_p$ and $(R)_p$ are unique among the class of solutions satisfying $\NN(\nabla u)\in L^p(\rn)$, and that our solutions to $(D)_q$ (resp. $(D)_{\lb}$) are unique among the class of solutions satisfying $N_*(u)\in L^q(\rn)$ (resp. $t\nabla u \in T^\infty_2(\reu)$, if $\beta=0$, or $u \in \dot{C}^\beta(\rn)$, if $\beta>0$). We refer the reader to \cite{HMaMo} for a proof of the uniqueness for $(N)_p$, $(R)_p$ and $(D)_q$ for a more general class of operators. Also, we refer the reader to~\cite{BM} for a proof of the uniqueness for $(D)_{\lb}$ in the case $\beta>0$. Finally, the uniqueness for $(D)_{\Lambda^0}=(D)_{BMO}$ follows by combining Theorem~\ref{P-bdd1} and Corollaries~\ref{r4.2} and~\ref{cor4.47} with the uniqueness result in Proposition~\ref{prop-BMO-unique} below (see also Remark~\ref{rem-BMO-Thm1.1}).
\end{proof}

We conclude this section by proving a uniqueness result for $(D)_{\Lambda^0}=(D)_{BMO}$.

\begin{proposition}\label{prop-BMO-unique}
Suppose that the hypotheses of Theorem \ref{th2} hold. If $Lu=0$ in~$\reu$ with
\begin{align}\label{eqb1}
\sup_{t>0}\|u(\cdot,t)\|_{BMO(\rn)} & <\infty,\\[4pt]
\label{eq2}
\|u\|_{BMO(\reu)} & <\infty,
\end{align}
and $u(\cdot,t)\to 0$ in the weak* topology on $BMO(\rn)$
as $t\to 0^+$, then $u= 0$ in~$\reu$ in the sense of $BMO(\rn)$. The analogous results hold for $L^*$ and in the lower half-space.
\end{proposition}

\begin{remark}\label{rem-BMO-Thm1.1} We note that \eqref{eq2} follows from the Carleson
measure estimate 
\begin{equation}\label{eqb3}
\sup_Q\frac1{|Q|}\dint_{R_Q} |\nabla u(x,t)|^2\, t \,dx dt < \infty\,,
\end{equation}
by the Poincar\'e inequality of \cite{HuS}, but we will not make explicit use
of \eqref{eqb3} in the proof of Proposition~\ref{prop-BMO-unique}.
\end{remark}

\begin{proof}
For each $\eps>0$, we set $u_\eps(x,t):=u(x,t+\eps)$ and $f_\eps:= u_\eps(\cdot,0)=u(\cdot,\eps)\,$. First, note that $u_\eps\in \dot{C}^\beta(\overline{\reu})$ for $0<\beta\leq\alpha$, with a bound that depends on $\eps$, where $\alpha$ is the De Giorgi/Nash exponent for $L$ in \eqref{eq2.DGN}.  To see this, we use the ``mean oscillation" characterization of $\dot{C}^\beta$ due to N. Meyers (see \eqref{eq4.45}). In particular, for $(n+1)$-dimensional boxes $I$ with side length $\ell(I)\leq \eps/2$, we use the DG/N estimate \eqref{eq2.DGN}, and for 
boxes with $\ell(I) \geq \eps/2$, we use \eqref{eq2}. We omit the routine details. By the uniqueness of $(D)_{\lb}$ for $\beta>0$ (see~\cite{BM}), we must then have
\begin{equation}\label{epseqn}
u_\eps(\cdot,t) = \pt f_\eps :=\mathcal{D}_L \big((-1/2)I +K_{L}\big)^{-1}f_\eps,\qquad\forall\,\eps>0\,.
\end{equation}

Next, by \eqref{eqb1}, we have $\sup_{\varepsilon>0} \|f_\varepsilon\|_{BMO(\rn)} < \infty$, and so there exists a subsequence $f_{\varepsilon_k}$ converging in the weak$^*$ topology on $BMO(\rn)$ to some $f$ in $BMO(\rn)$.
Let $g$ denote a finite linear combination of $H^1$-atoms, and for each $t>0$, 
set $g_t := \adj(\pt) g$, where $\adj$ denotes the $n$-dimensional Hermitian adjoint.
Let $\langle \cdot,\cdot \rangle$ denote the (complex-valued) dual pairing between $BMO(\rn)$ and $H^1_{\mathrm{at}}(\rn)$. Then, since $\adj(\pt)$ is bounded on $H^1(\rn)$, uniformly in $t>0$, by Theorem~\ref{P-bdd1}, we have
\begin{align*}
\int_{\mathbb{R}^n} (\pt f)\,\overline{g} &= \langle f\,,\, g_t \rangle = \lim_{k \to \infty} \langle f_{\eps_k}\,,\, g_t \rangle\\
&= \lim_{k \to \infty} \int_{\mathbb{R}^n}(\pt f_{\varepsilon_k}) \,\overline{g} 
= \lim_{k \to \infty} \int_{\mathbb{R}^n} u(\cdot,t+\varepsilon_k)\,\overline{g}
=\int_{\mathbb{R}^n} u(\cdot,t)\, \overline{g}\,,
\end{align*}
where in the next-to-last step we used~\eqref{epseqn}, and in the last step we used the DG/N estimate \eqref{eq2.DGN}, and the fact that $g$ is a finite linear combination of atoms.
Since $g$ was an arbitrary element of a dense subset of $H^1(\rn)$, this shows that $u(\cdot,t) = \pt f\,$. 

Now, since $u(\cdot,t) =\pt f$ for some $f$ in $BMO(\rn)$, by Corollary~\ref{cor4.47}, we have $u(\cdot,t)\to f$ in the weak* topology as $t\rightarrow 0^+$.  On the other hand, we assumed that $u(\cdot,t)\to 0$ in the weak* topology, thus $f=0$, and so $u(\cdot,t) = \pt f\, = 0$ in the sense of $BMO(\rn)$.
\end{proof}

\section{Boundary behavior of solutions}\label{s6}
\setcounter{equation}{0}
In this section, we present some a priori convergence results of ``Fatou-type", 
which show that Theorem \ref{th2} is optimal, in the sense that, necessarily, the data must belong
to the stated space, in order to obtain the 
desired quantitative estimate for the  
solution or its gradient.  The results also show that 
in some cases, our solutions enjoy
convergence to the data in a stronger sense than that provided by Corollary \ref{cor4.47}.
The results are contained in three lemmata. The first two results below are for the Neumann and Regularity problems.

\begin{lemma}\label{lemma6.1} Let $n/(n+1)<p<\infty$. Suppose that $L$ and $L^*$ satisfy the standard assumptions. If $Lu=0$ in $\RR^{n+1}_\pm$ and $\widetilde{N}_*^\pm(\nabla u) \in L^p(\RR^n)$, then the co-normal derivative $\partial_\nu u(\cdot,0)$ exists in the variational sense and belongs to $H^p(\RR^n)$, i.e., there exists a unique $g\in H^p(\RR^n)$, and we set $\partial_\nu u(\cdot,0):=g$, with 
\begin{equation}\label{eqb2}
\|g\|_{H^p(\rn)}\leq C\, \|\NN^\pm(\nabla u)\|_{L^p(\rn)}\,,
\end{equation}
such that
\begin{equation}\label{eq5.2}
\int_{\RR^{n+1}_\pm} A\nabla u\cdot \nabla\Phi\, dX = \pm\langle g\,, \Phi (\cdot,0)\rangle\,,\qquad\forall\, \Phi \in C^\infty_0(\ree)\,,
\end{equation}
where $\langle g\,, \Phi (\cdot,0)\rangle :=\int_{\rn} g(x)\, \Phi(x,0)\, dx$, if $p\geq 1$, and 
$\langle g\,, \Phi (\cdot,0)\rangle$ denotes the usual pairing of the distribution $g$ with the test function
$\Phi(\cdot,0)$, if $p<1$. Moreover, there exists a unique $f\in \dot{H}^{1,p}(\rn)$, and we set $u(\cdot,0):= f$, with 
\begin{equation}\label{eq6.3*}
\|f\|_{\dot{H}^{1,p}(\rn)}\leq 
C\|\widetilde{N}_*^\pm(\nabla u)\|_{L^p(\rn)}\,,
\end{equation} 
such that $u\to f$ non-tangentially.
\end{lemma}

\begin{lemma}\label{lemma6.2a}  Suppose that $L$ and $L^*$ satisfy the standard assumptions. Suppose also that $Lu=0$ in $\RR^{n+1}_\pm$ and $\widetilde{N}_*^\pm(\nabla u) \in L^p(\RR^n)\,$ for some $n/(n+1)<p<\infty$. There exists $\epsilon>0$, depending only on the standard constants, such that in the case $1<p<2+\epp$, one has
\begin{equation}\label{eq6.4*}
\sup_{\pm t>0} \|\nabla u(\cdot,t)\|_{L^p(\rn)} \leq C_p \|\NN^\pm(\nabla u)\|_{L^p(\rn)}\,,
\end{equation}
\begin{equation}\label{eq6.5*}
-e_{n+1} \cdot A\nabla u(\cdot,t) \to \partial_\nu u(\cdot,0) \textrm{ weakly in } L^p \textrm{ as } t\to 0^\pm\,,
\end{equation}
\begin{equation}\label{eq6.6*}
\nabla_x u(\cdot,t) \to \nabla_x u(\cdot,0) \textrm{ weakly in } L^p \textrm{ as } t\to 0^\pm\,,
\end{equation}
where $\partial_\nu u(\cdot,0) \in L^p(\rn)$ and $u(\cdot,0)\in \dot{L}^p_1(\rn)$ denote the variational co-normal and non-tangential boundary trace, respectively, defined in Lemma \ref{lemma6.1}.

Also, in the case $n/(n+1)<p\leq 1$, if there exists $h\in \dot{H}^{1,p}(\rn)$ such that $\nabla_x u(\cdot,t) \to \nabla_x h$ in the sense of tempered distributions, then $u(\cdot,0)=h$ in the sense of $\dot{H}^{1,p}(\rn)$, where $u(\cdot,0)\in \dot{H}^{1,p}(\rn)$ denotes the non-tangential boundary trace defined in Lemma \ref{lemma6.1}.
\end{lemma}

The third and final result below is for the Dirichlet problem.
\begin{lemma}\label{lemma6.2}  Suppose that  the hypotheses of Theorem \ref{th2} hold. Let $2-\epp<p<\infty$ denote the range of well-posedness of $(D)_p$. If $Lu=0$ in $\mathbb{R}^{n+1}_\pm$ and
\begin{equation}\label{eq6.3}
\|N_*^\pm(u)\|_{L^p(\rn)}<\infty\,,
\end{equation}
then there exists a unique $f\in L^p(\rn)$, and we set $u(\cdot,0):=f$, such that
\begin{equation}\label{eq6.4}
u \to f \textrm{ non-tangentially, and } u(\cdot,t) \to f \textrm{ in } L^p(\rn)\textrm{ as } t\to 0^\pm\,.
\end{equation}
\end{lemma}

\begin{proof}[Proof of Lemma \ref{lemma6.1}]  We suppose that $\NN(\nabla u) \in L^p(\rn)$, 
and we seek  a variational co-normal $g\in H^p(\rn)$, and a
non-tangential limit $f \in \dot{H}^{1,p}(\rn)$, 
satisfying the bounds  \eqref{eqb2} and \eqref{eq6.3*}.  
The case $p>1$ may be obtained by following, {\it mutatis mutandi},
the proof of \cite[Theorem 3.1]{KP} (see also \cite[Lemma 4.3]{AAAHK},
stated in this paper as Lemma \ref{l4.1}, which treats the case $p=2$ by following
\cite[Theorem 3.1]{KP}). We omit the details. The case $p\leq 1$, which is a bit more problematic, is treated below.  

First, we consider the existence of the non-tangential limit $f\in \dot{H}^{1,p}(\rn)$,
assuming now that  $\NN(\nabla u) \in L^p(\rn)$
with $n/(n+1)<p\leq 1$.  In fact, following the proof of \cite[Theorem~3.1, p. 462]{KP}),
we see that the non-tangential limit $f(x)$ exists at every point $x\in\rn$ for which
$\NN(\nabla u)(x)$ is finite (thus, a.e. in $\rn$, no matter the value of $p$), 
and moreover, for any pair of points $x,y\in\rn$
at which  $\NN(\nabla u)(x)$ and $\NN(\nabla u)(y)$ are finite, we have the pointwise estimate
\[
|f(x) -f(y)| \,\leq \,C\,|x-y| \left(\NN(\nabla u)(x) \,+\,\NN(\nabla u)(y)\right)\,,
\]
where $C$ depends only on the standard constants.  Thus, by the criterion of \cite{KoS}, 
we obtain immediately that $f\in \dot{H}^{1,p}(\rn)$, with $\|f\|_{\dot{H}^{1,p}(\rn)} \lesssim \|\NN(\nabla u)\|_{L^p(\rn)}$.

Next, we consider the existence of the co-normal derivative $g\in H^p(\RR^n)$. We use $\langle \cdot,\cdot\rangle$ to denote the usual pairing of tempered distributions $\s'(\mathbb{R}^d)$ and Schwartz functions $\s(\mathbb{R}^d)$, where  $d$ may be either $n$ or $n+1$ (the usage will be clear from the context). By Lemma~\ref{L-improve}, for all $0<q \leq 2n/(n+1)$, we have
\begin{equation}\label{eq3}
\|\nabla u\|_{L^{q(n+1)/n}(\reu)} \leq C(q,n)\, \|\NN(\nabla u)\|_{L^q(\rn)}\,,
\end{equation}
and since $\NN(\nabla u)\in L^p(\rn)$, this implies that $\nabla u \in L^r(\reu)$, with $r:= p(n+1)/n >1.$
We may then define a linear functional $\Lambda=\Lambda_u \in \s'(\ree)$ by
$$\langle\Lambda,\Phi\rangle:= \iint_{\reu} A\nabla u\cdot
\nabla\Phi\,,\qquad \forall\,\Phi \in \s(\ree)\,.$$
For $\vp\in\s(\rn)$, we say that $\Phi\in \s(\ree)$ is an extension of $\vp$ if $\Phi(\cdot,0)=\vp$.
We now define a linear functional $g \in \s'(\rn)$ by setting
$$\langle g,\vp\rangle:= \langle\Lambda,\Phi\rangle\,,\qquad \forall\,\varphi \in \s(\rn)\,,$$
where $\Phi$ is any extension of $\vp$.  Since such an extension of $\varphi$ need not be unique, however, we must verify that $g$ is well-defined.  To this end, fix $\vp\in\s(\rn)$,
and let $\Phi_1,\Phi_2\in\s(\ree)$ denote any two extensions
of $\vp$.  Then $\Psi:= \Phi_1-\Phi_2\in \s(\ree),$ with $\Psi(\cdot,0)\equiv 0$, and so
$ \langle\Lambda,\Psi\rangle = 0$, by the definition of a (weak) solution.  Thus, the linear functional
$g$ is well-defined, and so $u$ has a variational co-normal $\pun(\cdot,0): = g$ in $\s'(\rn)$ satisfying ~\eqref{eq5.2}.

It remains to prove \eqref{eqb2}. For $\vp\in\s(\rn)$, we set $M_\vp f:=\sup_{t>0}|\vp_t * f|\,,$ where as usual $\vp_t(x):=t^{-n}\vp(x/t)$.  We recall that a tempered distribution $f$ belongs to
$H^p(\rn)$ if and only if $M_\vp f \in L^p(\rn)$, for some $\vp\in\s(\rn)$ with $\int_{\rn}\vp = 1$
(see, e.g., \cite[Theorem~1, p.~91]{St2}),
and we have the equivalence $\|f\|_{H^p(\rn)} \approx \|M_\vp f\|_{L^p(\rn)}\,$. We now fix $\vp \in C_0^\infty(\rn)$, with $\vp\geq0$, $\int\vp = 1$, and $\supp\vp\subset \Delta(0,1):=\{x\in\rn:|x|<1\}$, so we have
$$\|\pun\|_{H^p(\rn)}\leq C\, \|M_\vp(\pun)\|_{L^p(\rn)}\,,$$
and it suffices to show that 
\[
\|M_\vp(\pun)\|_{L^p(\rn)}  \leq C \|\NN(\nabla u)\|_{L^p(\rn)} \,.
\]
We claim that
\begin{equation}\label{eq5}
M_\vp(\pun)\leq C \left(M\left(\NN(\nabla u)\right)^{n/(n+1)}\right)^{(n+1)/n}\,,
\end{equation}
pointwise, where $M$ denotes the usual Hardy-Littlewood maximal operator.
Taking the claim for granted momentarily, we see that
\begin{equation*}
\int_{\rn}M_\vp(\pun)^p \lesssim\int_{\rn}\left(M\left(\NN(\nabla u)\right)^{n/(n+1)}\right)^{p(n+1)/n}
\lesssim \int_{\rn}\left(\NN(\nabla u)\right)^{p}\,,
\end{equation*}
as desired, since $p(n+1)/n>1$.  

It therefore remains to establish \eqref{eq5}.  To this end, we fix $x\in \rn$ and $t>0$, 
set $B:= B(x,t):=\{Y\in \ree: |Y-x|<t\}$, and fix a smooth cut-off function
$\eta_{B}\in  C_0^\infty(2B)$, with
$\eta_{B}\equiv 1$ on $B$, $0\leq\eta_{B}\leq 1,$
and $|\nabla \eta_{B}|\lesssim 1/t$.   Then
$$\Phi_{x,t}(y,s):= \eta_B(y,s)\, \vp_t(x-y)$$
is an extension of $\vp_t(x-\cdot)$, 
with $\Phi_{x,t}\in C_0^\infty(2B),$ which satisfies
$$0\leq \Phi_{x,t} \lesssim t^{-n}\,,\qquad|\nabla_Y \Phi_{x,t}(Y)|\lesssim t^{-n-1}\,.$$
We then have
\begin{align*}
|\left(\vp_t*\pun\right)&(x)|=|\langle \pun,\vp_t(x-\cdot)\rangle| =\left|
 \iint_{\reu} A\nabla u\cdot
 \nabla\Phi_{x,t}\right|\\[4pt]
&\lesssim \,t^{-n-1}\iint_{\reu\cap 2B} |\nabla u|\,
 \lesssim \,t^{-n-1}\left(\int_{\rn}\left(\NN(|\nabla u| 1_{2B})(y)\right)^{n/(n+1)}dy\right)^{(n+1)/n}\,,
\end{align*}
where in the last step we have used \eqref{eq3} with $q = n/(n+1)$. For $C>0$ chosen sufficiently large, simple geometric considerations then imply that
$$\NN(|\nabla u| 1_{2B})(y) \leq \NN(\nabla u)(y)\,1_{\Delta(x,Ct)}(y)\,,$$
where $\Delta(x,Ct):=\{y\in\rn:|x-y|<Ct\}.$  Combining the last two estimates,
we obtain
$$|\left(\vp_t*\pun\right)(x)|\lesssim 
\left(t^{-n}\int_{|x-y|<Ct}\left(\NN(\nabla u)(y)\right)^{n/(n+1)}dy\right)^{(n+1)/n}\,.$$
Taking the supremum over $t>0$, we obtain \eqref{eq5}, as required.
\end{proof}

\begin{proof}[Proof of Lemma \ref{lemma6.2a}]
We begin with  \eqref{eq6.4*} and follow the proof in the case $p=2$ from~\cite{AAAHK}.  The desired bound for $\partial_t u$ follows readily from
$t$-independence and the Moser local boundedness estimate
\eqref{eq2.M}.  Thus, we only need to consider
$\nabla_x u$. Let $\vec{\psi} \in C_0^\infty (\mathbb{R}^n,\mathbb{C}^n)$, with $\|\vec{\psi}\|_{p'} =1$.
For $t>0$, let $\dd_t$ denote the grid of dyadic cubes $Q$ in $\rn$ with side length satisfying $\ell(Q) \leq t < 2\ell(Q)$,
and for $Q\in \dd_t$, set $Q^*:= 2Q\times (t/2,3t/2)$.
Then, using the Caccioppoli-type estimate on 
horizontal slices in \cite[(2.2)]{AAAHK}, we obtain
\begin{align*}
\bigg|\int_{\rn} \nabla_y  u(y,t) \cdot\vec{\psi}(y) \, dy\bigg| & \leq \,\left(\int_{\rn} |\nabla_y u(y,t)|^p \,dy\right)^{1/p} \|\vec{\psi}\|_{p'} \\[4pt]
&= \,\left(\sum_{Q\in\dd_t} \frac{1}{|Q|} \int_Q \int_{Q} |\nabla_y u(y,t)|^pdy \,dx\right)^{1/p} \\[4pt]
&\lesssim\, \left(\sum_{Q\in\dd_t}\int_Q\left( \frac1{|Q^*|}\iint_{Q^*} 
|\nabla_y u(y,s)|^2\, dyds\right)^{p/2}dx\right)^{1/p} \\[4pt]
&\lesssim\,\left(\int_{\rn}\Big(\NN(\nabla u)\Big)^p\right)^{1/p}\,.
\end{align*}
This concludes the proof of \eqref{eq6.4*}.

Next, we prove \eqref{eq6.5*}.  By  \eqref{eqb2}  and  \eqref{eq6.4*}, and the density 
of $C^\infty_0(\rn)$ in $L^{p'}(\rn)$,
it is enough to prove that
\[
\lim_{t\to 0}\int_{\rn} \vec{N}\cdot A(x) \nabla u(x,t)\, \phi(x)\, dx \,=\, \int_{\rn} \partial_\nu u(x,0)\, \phi(x)\, dx\,,\quad \forall\,\phi\in C^\infty_0(\rn)\,,
\]
where $\vec{N}:=-e_{n+1}$.  For $\phi\in C^\infty_0(\rn)$, let $\Phi$ denote a $C_0^\infty(\ree)$ extension of $\phi$ to $\ree$, so by \eqref{eq5.2}, it suffices to show that 
 \begin{equation}\label{eq6.16a} 
 \lim_{t\to 0}\int_{\mathbb{R}^n} \vec{N}\cdot A\nabla u(\cdot,t)\,\phi
 \,=\, \iint_{\mathbb{R}^{n+1}_+} A \nabla u \cdot \nabla \Phi \,.
 \end{equation}
 Let $P_\eps$ be an approximate identity in $\rn$ with a smooth, compactly supported 
 convolution kernel.
 Integrating by parts, we see that for each $\varepsilon > 0$,
 \begin{equation}\label{eq4.nt5}\int_{\mathbb{R}^n} \vec{N}\cdot 
 P_\varepsilon(A\nabla u(\cdot,t))\,\phi
 = \iint_{\mathbb{R}^{n+1}_+}P_\varepsilon\left( A \nabla u(\cdot,t+s)\right)(x) 
 \cdot \nabla \Phi (x,s)\, dx ds ,\end{equation}
 since $Lu=0$ and our coefficients are $t$-independent.   By the dominated convergence theorem,
we may pass to the limit as $\varepsilon \to 0$ in \eqref{eq4.nt5} to obtain
  \begin{align*}
  \int_{\mathbb{R}^n} \vec{N}\cdot A\nabla u(\cdot,t)\,\phi
 &= \iint_{\mathbb{R}^{n+1}_+} A(x) \nabla u(x,t+s) 
 \cdot \nabla \Phi (x,s)\, dx ds\\[4pt]
 &=\int_t^\infty\!\!\!\int_{\mathbb{R}^{n}} A(x) \nabla u(x,s) 
 \cdot \nabla \Big(\Phi (x,s-t)-\Phi(x,s) \Big) \, dx ds\\[4pt] 
 &\qquad+\, \int_t^\infty\!\!\!\int_{\mathbb{R}^{n}} A(x) \nabla u(x,s) 
 \cdot \nabla\Phi(x,s) \, dx ds\,=:I(t) + II(t)\,.
  \end{align*}
By Lemma \ref{L-improve} and the dominated convergence theorem, we have $I(t) \to 0$, as $t\to 0$, and $II(t) \to \iint_{\mathbb{R}^{n+1}_+} A \nabla u \cdot \nabla \Phi $, as $t\to 0$, hence \eqref{eq6.16a} holds.

Next, we prove \eqref{eq6.6*}.  By \eqref{eq6.3*} and \eqref{eq6.4*},
and the density of $C^\infty_0(\rn)$ in $L^{p'}(\rn)$,
it is enough to prove that
\begin{equation}\label{eq6.17} 
\lim_{t\to 0}\int_{\rn} u(x,t)\, \dv_x \vec{\psi}(x)\, dx \,=\, \int_{\rn} u(x,0)\, \dv_x\vec{\psi}(x)\, dx\,,\quad \forall\,\vec{\psi}\in C^\infty_0(\rn,\mathbb{C}^n)\,.
\end{equation}
Following the proof of \cite[Theorem 3.1, p. 462]{KP}),
we obtain
\begin{equation}\label{neednow}
|u(x,t)-u(x,0)| \leq C t \NN(\nabla u)(x)\,,\quad\text{ for a.e. }x\in \rn\,,
\end{equation}
whence \eqref{eq6.17} follows.

Finally, we consider the case $n/(n+1)<p\leq 1$, and we assume there exists $h\in \dot{H}^{1,p}(\rn)$ such that $\nabla_x u(\cdot,t) \to \nabla_x h$ in the sense of tempered distributions. By Sobolev embedding, $u(\cdot,0)$ and $u(\cdot,t)$ belong (uniformly in $t$) to $L^{q}(\rn)$, 
with $1/q = 1/p - 1/n$.  Note that $q>1$, since $p>n/(n+1)$. For all $\eps\in (0,1)$, by combining the pointwise estimate~\eqref{neednow}, which still holds in this case, with the trivial bound $|u(\cdot,t)-u(\cdot,0)|\leq |u(\cdot,t)| +|u(\cdot,0)|$, we obtain
$$|u(x,t)-u(x,0)| \leq C \left(t \widetilde{N}_*(\nabla u)(x)\right)^\eps \big(|u(x,t)| +|u(\cdot,0)(x)|\big)^{1-\eps}\,,\quad\text{ for a.e. }x\in \rn\,.$$
For $p,q$ as above, set $r=q/(1-\eps)$, $s=p/\eps$, and choose $\eps\in (0,1)$,
depending on $p$ and $n$, so that $1/r+1/s=1$.  Then by H\"older's inequality, for all $\psi\in\s(\rn)$, we have
\begin{align}\begin{split}\label{eqa}
\bigg|\int_{\rn} \big(u(x,t) -&u(x,0)\big)\psi(x)\, dx\bigg| \\[4pt]
&\lesssim \,\|\psi\|_\infty \,t^\eps \left(\int_{\rn} \left(\widetilde{N}_*(\nabla u)\right)^p\right)^{1/s}
\left(\|u(\cdot,0)\|_q + \sup_{t>0} \|u(\cdot,t)\|_q\right)^{1-\eps}\,\to \,0\,,
\end{split}\end{align}
as $t\to 0$.  On the other hand, for all $\vec{\phi}\in \s(\rn,\mathbb{C}^n)$, we have
$$\int_{\rn} \big(u(x,t) -h(x)\big)\dv_x\vec{\phi}(x)\, dx \to 0\,.$$
Combining the latter fact with \eqref{eqa}, applied with $\psi =\dv_x \vec{\phi}$, we obtain
$$\int_{\rn} h(x)\dv_x\vec{\phi}(x)\, dx = \int_{\rn} u(x,0)\dv_x\vec{\phi}(x)\, dx\,,\qquad \forall \vec{\phi} \in \mathcal{S}(\rn,\mathbb{C}^n)\,,$$
thus $\nabla_x h = \nabla_x u(\cdot,0)$ as tempered distributions, and since each belongs to $H^p(\rn)$, we also have $\nabla_x h = \nabla_x u(\cdot,0)$ in $H^p(\rn)$, hence $u(\cdot,0)=h$ in the sense of $\dot{H}^{1,p}(\rn)$.
\end{proof}

\begin{proof}[Proof of Lemma \ref{lemma6.2}]
We first prove that~\eqref{eq6.4} holds in the case that
$u=\mathcal{D}h$ for some $h\in L^p(\rn)$.  
Indeed, in that scenario, the case $p=2$ 
has been treated in \cite[Lemma 4.23]{AAAHK}.  To handle the remaining range of $p$,
we observe that by Theorem~\ref{P-bdd1}, we have
\[
\|N_*(\mathcal{D}h)\|_{L^p(\rn)} \leq\, C_p\, \|h\|_{L^p(\rn)}\,.
\]
We may  therefore exploit the usual technique, whereby a.e. convergence on a dense class (in our case $L^2\cap L^p$), along with
$L^p$ bounds on the controlling maximal operator, imply a.e. convergence for all 
$h\in L^p(\rn)$.   We omit the standard argument.  Convergence in $L^p(\rn)$ 
then follows by the dominated convergence theorem.

Thus, it is enough to show that
$u = \mathcal{D}h$ for some $h\in L^p(\mathbb{R}^n).$  We follow
the corresponding argument for the case $p=2$ given in \cite{AAAHK}, which in
turn follows \cite[pp.~199--200]{St1}, substituting $\mathcal{D}$ for the classical Poisson kernel.
For each $\varepsilon > 0$, set $f_\varepsilon := u(\cdot,\varepsilon)$, and let $u_\varepsilon  := \mathcal{D}\big((-1/2)I + K \big)^{-1} f_\varepsilon$ denote the layer potential solution with data $f_\varepsilon$. We claim that $u_\varepsilon(x,t) = u(x,t+\varepsilon).$ To prove this, we set $U_\varepsilon(x,t) :=  u(x,t + \varepsilon) - u_\varepsilon(x,t)$, and observe that
\begin{enumerate}
\item[(i)] $LU_\varepsilon = 0$ in $\mathbb{R}^{n+1}_+$ (by $t$-independence of coefficients).
\item[(ii)] Estimate \eqref{eq6.3} holds for $U_\varepsilon$, uniformly in
$\varepsilon > 0$. 
\item[(iii)] $U_\varepsilon (\cdot,0) = 0$ and $U_\varepsilon(\cdot,t) \to 0 $ non-tangentially 
and in $L^p$, as $t\to 0$.
\end{enumerate}
Item (iii) relies on interior continuity \eqref{eq2.DGN} and smoothness in $t$, along with the
result for layer potentials noted above. The claim then follows by the uniqueness for $(D)_p$, which is proved in~\cite{HMaMo} for a more general class of operators.

We now complete the proof of the lemma.  
For convenience of notation, for each $t>0$, we set $\mathcal{D}_t h:=\mathcal{D} h(\cdot,t)$.
By \eqref{eq6.3},
$\sup_\varepsilon \|f_\varepsilon\|_{L^p(\rn)} < \infty$, and so there exists a subsequence
$f_{\varepsilon_k}$ converging in the weak$^*$ topology on $L^p(\rn)$ to some $f \in L^p(\rn)$.
For each $g \in L^{p'}(\rn)$, we set $g_1 := \adj\big((-1/2)I+ K \big)^{-1}
\adj (\mathcal{D}_t) g$, and observe that
\begin{align*}
\int_{\mathbb{R}^n} \left[\mathcal{D}_t \big((-1/2)I+ K \big)^{-1} f \right]\,\overline{g}
& = \int_{\mathbb{R}^n} f\, \overline{g_1} = \lim_{k \to \infty} \int_{\mathbb{R}^n} f_{\varepsilon_k}
\overline{g_1}\\
&= \lim_{k \to \infty} \int_{\mathbb{R}^n}\left[\mathcal{D}_t \big((-1/2)I+ K \big)^{-1}
\!f_{\varepsilon_k}\right] \,\overline{g} \\
&=\lim_{k \to \infty} \int_{\mathbb{R}^n} u(\cdot,t+\varepsilon_k)\,\overline{g}
=\int_{\mathbb{R}^n} u(\cdot,t)\, \overline{g}.
\end{align*}
It follows that $u=\mathcal{D} h$, with $h= \big((-1/2)I +K\big)^{-1} f$ in $L^p(\rn)$, as required.
\end{proof}

\section{Appendix: Auxiliary lemmata}
\setcounter{equation}{0}
We now return to prove some technical results that were used to prove Proposition~\ref{r4.1} and Lemmata~\ref{lemma6.1}-\ref{lemma6.2a}. The results are stated in the more general setting of a Lipschitz graph domain of the form $\Omega:=\{(x,t)\in\ree:\, t>\phi(x)\}$, where $\phi:\rn\to \re$ is Lipschitz. We set $M:=\|\nabla\phi\|_{L^\infty(\rn)}<\infty$, and consider constants
\begin{equation}\label{eta-alpha}
0<\eta<\frac{1}{M},\qquad
0<\beta<\min\,\Bigl\{1,\frac{1}{M}\Bigr\}.
\end{equation}
We define the cone 
\[
\Gamma:=\{X=(x,t)\in\RR^{n+1}:\,|x|<\eta t\}.
\]
For $X\in\Omega\subset\ree$, we use the notation $\delta(X):=\dist(X,\partial\Omega)$. 
For $u\in L^2_{\mathrm{loc}}(\Omega)$, we set
\begin{equation}\label{eq2.3}
\widetilde{N}_*(u)(Q):=\sup\limits_{X\in Q+\Gamma}\left(
\fint_{B(X,\,\beta\delta(X))}|u(Y)|^2\,dY\right)^{1/2}, 
\qquad Q\in\partial\Omega\,,
\end{equation}
and
\begin{equation}\label{eq2.3*}
N_*(u)(Q):=\sup\limits_{X\in Q+\Gamma}\,|u(X)|\,,
\qquad Q\in\partial\Omega.
\end{equation}
If we want to emphasize the dependence on $\eta$ and $\beta$, then we shall write $\Gamma_{\eta}$, $\widetilde{N}_{*,\eta,\,\beta}$, $N_{*,\eta}$. The lemma below shows that the choice of $\eta$ and $\beta$, within the permissible range in \eqref{eta-alpha}, is immaterial for $L^p(\partial\Omega)$ estimates of $\widetilde{N}_{*,\eta,\,\beta}$.

\begin{lemma}\label{L-cones}
Let $\Omega\subset\mathbb{R}^{n+1}$ denote a Lipschitz graph domain. For each $p\in(0,\infty)$ and 
\[
0<\eta_1,\,\eta_2<\frac{1}{M},\qquad 
0<\beta_1,\,\beta_2<\min\Bigl\{1\,,\,\frac{1}{M}\Bigr\},
\]
there exist constants 
$C_j=C_j(M,p,\eta_1,\eta_2,\beta_1,\beta_2)\in(0,\infty)$, $j=1,2$, such that 
\begin{equation}\label{equiv-N}
C_1\|\widetilde{N}_{*,\eta_2,\,\beta_2}u\|_{L^p(\partial\Omega)}
\leq \|\widetilde{N}_{*,\eta_1,\,\beta_1}u\|_{L^p(\partial\Omega)}
\leq C_2\|\widetilde{N}_{*,\eta_2,\,\beta_2}u\|_{L^p(\partial\Omega)}
\end{equation}
for all $u \in L^2_{\mathrm{loc}}(\Omega)$. 
\end{lemma}

\begin{proof}
First, a straightforward adaptation
of the argument in \cite[p.~62]{St2} gives 
\begin{equation}\label{equiv-N2}
\|\widetilde{N}_{*,\eta_2,\,\beta}u\|_{L^p(\partial\Omega)}
\leq C\|\widetilde{N}_{*,\eta_1,\,\beta}u\|_{L^p(\partial\Omega)}, 
\end{equation}
 whenever $0<\eta_1<\eta_2<1/M$, $p\in(0,\infty)$ and
$\beta\in(0,1)$. The opposite inequality is trivially true (with $C=1$). 
Thus, since $\frac{1-M\eta}{M+\eta}\nearrow\frac{1}{M}$ as $\eta\searrow 0$, estimate 
(\ref{equiv-N}) will follow as soon as we prove that for any 
\begin{equation}\label{ineq-2}
0<\eta<\frac{1}{M},\qquad 
0<\beta_1<\beta_2<\min\Bigl\{1\,,\,\frac{1-M\eta}{M+\eta}\Bigr\},
\end{equation}
there exists a finite constant 
$C=C(M,p,\eta,\beta_1,\beta_2)>0$ such that 
\[
\|\widetilde{N}_{*,\eta,\,\beta_2}u\|_{L^p(\partial\Omega)}
\leq C\|\widetilde{N}_{*,\eta,\,\beta_1}u\|_{L^p(\partial\Omega)}
\]
for all $u\in L^2_{\mathrm{loc}}(\Omega)$. 
To this end, let $\eta$, $\beta_1$, $\beta_2$ be as in (\ref{ineq-2}) 
and consider two arbitrary points, $Q\in\partial\Omega$ and 
$X\in Q+\Gamma_{\eta_2}$, as well as  
two parameters, $\beta'\in (0,\beta_2)$ and $\varepsilon>0$, to be chosen later. The parameter $\varepsilon>0$ and Euclidean geometry ensure that
\begin{align*}\begin{split}
|X-&Y| < \beta_2 \delta(X) \Longrightarrow \\[4pt] &|B(X,\beta'\delta(X))| \leq C(n,\beta_2,\beta',\varepsilon) \, |B(Y,(\beta_2-\beta'+\varepsilon)\delta(X)) 
\cap B(X,\,\beta'\delta(X))|.
\end{split}\end{align*}
We also have 
\begin{equation}\label{trivial-1}
|X-Z|<\beta'\,\delta(X)\Longrightarrow
\frac{1}{1+\beta'}\delta(Z)\leq\delta(X)\leq\frac{1}{1-\beta'}\delta(Z), 
\end{equation}
and 
\[
B(X,\beta'\delta(X))\subset Q+\Gamma_\kappa,\quad\mbox{where }\,
\kappa:=\frac{\eta+\beta'}{1-\beta'\eta}. 
\]
Note that, due to our assumptions, $0<\kappa<{1}/{M}$. 
Using Fubini's Theorem and the preceding considerations, we may then write 
\begin{align*}
\hspace{1cm}&\hspace{-1cm}\fint_{B(X,\,\beta_2\delta(X))} |u(Y)|^2\,dY
= \fint_{B(X,\,\beta_2\delta(X))}
\left(\fint_{B(Y,\,(\beta_2-\beta'+\varepsilon)\delta(X)) 
\cap B(X,\,\beta'\delta(X))}1\, dZ\right)|u(Y)|^2\,dY\\
&=  C(n,\beta_2,\beta',\varepsilon)\fint_{B(X,\,\beta'\delta(X))}
\left(\fint_{B(X,\,\beta_2\delta(X))} 1_{B(Z,\,(\beta_2-\beta'+\varepsilon)\delta(X))}(Y)\, |u(Y)|^2\,dY\right)dZ\\
&=  C(n,\beta_2,\beta',\varepsilon)\fint_{B(X,\,\beta'\delta(X))}
\left(\fint_{B(Z,\,(\beta_2-\beta'+\varepsilon)\delta(X))}|u(Y)|^2\,dY\right)dZ\\
&\leq  C(n,\beta_2,\beta',\varepsilon)\fint_{B(X,\,\beta'\delta(X))}
\left(\fint_{B(Z,\,\frac{\beta_2-\beta'+\varepsilon}{1-\beta'}\delta(Z))}
|u(Y)|^2\,dY\right)dZ\\
&\leq  C(n,\beta_2,\beta',\varepsilon) \left(\widetilde{N}_{*,\kappa,\frac{\beta_2-\beta'+\varepsilon}{1-\beta'}}(u)(Q)\right)^2.
\end{align*}
We now choose $\varepsilon \in (0,\beta_1(1-\beta_2))$ and set
$\beta':=\frac{\beta_2-\beta_1+\varepsilon}{1-\beta_1}$ to ensure that $\beta'\in(0,\beta_2)$ and $\frac{\beta_2-\beta'}{1-\beta'}=\beta_1$, so the inequality above 
further yields 
\begin{equation}\label{geo-2}
\widetilde{N}_{*,\eta,\,\beta_2}(u)(Q)\leq C \widetilde{N}_{*,\kappa,\,\beta_1}(u)(Q)\quad
\mbox{for some }\,\,\kappa=\kappa(\beta_1,\beta_2,\eta)\in(0,1/M).
\end{equation}
Consequently, by (\ref{geo-2}) and (\ref{equiv-N2}), we have
\[
\|\widetilde{N}_{*,\eta,\,\beta_2}u\|_{L^p(\partial\Omega)}
\leq C\|\widetilde{N}_{*,\kappa,\,\beta_1}u\|_{L^p(\partial\Omega)}
\leq C\|\widetilde{N}_{*,\eta,\,\beta_1}u\|_{L^p(\partial\Omega)}.
\]
This finishes the proof of the lemma. 
\end{proof}

We now prove a self-improvement property for $L^p(\Omega)$ estimates of solutions.
\begin{lemma}\label{L-improve}
Let $\Omega\subset\mathbb{R}^{n+1}$ denote a Lipschitz graph domain. Suppose that $w\in L^2_{\mathrm{loc}}(\Omega)$, and that $\widetilde{N}_*(w)\in L^p(\partial\Omega)$ for 
some $p\in(0,\infty)$. First, if $0 < p \leq 2n/(n+1)$, then 
\begin{equation}\label{impr-1}
w\in L^{{p(n+1)}/{n}}(\Omega)\quad\mbox{and}\quad
\|w\|_{L^{{p(n+1)}/{n}}(\Omega)}
\leq C(\partial\Omega,p)\,\|\widetilde{N}_*(w)\|_{L^p(\partial\Omega)}.
\end{equation}
Second, if $0<p<\infty$, and if $Lw=0$ in $\Omega$, then \eqref{impr-1} holds.
Finally, there exists $q=q(n,\Lambda)>2$ such that if $0<p<qn/(n+1)$, and if $w=\nabla u$ for some 
solution $u\in L^2_{1,\,\mathrm{loc}}(\Omega)$ of $Lu=0$ in $\Omega$, then 
\eqref{impr-1} holds.
\end{lemma}

\begin{proof} 
Fix $\eta$, $\beta$ as in (\ref{eta-alpha}). 
We observe that by Lemma \ref{L-cones}, the choice of $\beta$ within the permissible range is immaterial.
We now choose $\beta'$ so that $0<\beta'<\beta/2<1/2$.  Then
\eqref{trivial-1} holds, and we have $\beta'/(1-\beta')<\beta$.

\smallskip

\noindent{\bf Case 1}. Suppose that $w\in L^2_{\mathrm{loc}}(\Omega)$, and that 
$\widetilde{N}_*(w)\in L^p(\partial\Omega)$ for some $0<p\leq2n/(n+1)$. To prove \eqref{impr-1}, we set 
\[
F(Z):=\left(\fint_{B(Z,\,\beta\delta(Z))}|w(X)|^2\,dX\right)^{1/2},
\qquad Z\in\Omega,
\]
and observe that
\begin{align*}\iint_{\Omega} |w(X)|^{p(n+1)/n}\, dX
&=\iint_{\Omega}
\left(\fint_{|X-Z|<\beta' \delta(X)}\,dZ\right)\,
|w(X)|^{p +p/n}\, dX\\[4pt]
&\leq C \iint_{\Omega}
\left(\fint_{|X-Z|<\beta \delta(Z)}\,
|w(X)|^{p +p/n}\,dX\right)\,dZ
\\[4pt]
&\leq C \iint_{\Omega}
\left(\fint_{|X-Z|<\beta \delta(Z)}\,
|w(X)|^2 dX\right)^{(p +p/n)/2}\, dZ\\[4pt]
&=:\iint_{\Omega}F(Z)^{p+p/n}\, dZ 
\leq C\, \|\mu\|_{\mathcal{C}}\int_{\partial\Omega}N_*(F)^p,
\end{align*}
where we have used Fubini's Theorem, \eqref{trivial-1} and the fact that $\beta'/(1-\beta')<\beta$
in the first inequality, the fact that
$p(n+1)/n \leq 2$ in the second, and
Carleson's lemma (which still holds in the present setting) in the third.
In particular, we are using $\|\mu\|_{\mathcal{C}}$ to denote the Carleson norm of the measure
$$d\mu(Z):= F(Z)^{p/n} \,1_\Omega(Z)\,dZ.$$
Also, by definition, $N_*(F) = \widetilde{N}_*(w)$ (cf. \eqref{eq2.3} and \eqref{eq2.3*}), and so
$$\|N_*(F)\|_{L^p(\partial \Omega)}=\|\widetilde{N}_*(w)\|_{L^p(\partial \Omega)}<\infty.$$  
Thus, to finish the proof of Case 1, it is enough to observe that for every ``surface ball" 
$\Delta(P,r):= B(P,r)\cap\partial\Omega\,$, where $P:=(x,\varphi(x))\in \partial\Omega$ and $r>0$, we have
\begin{align*}\frac{1}{|\Delta(P,r)|}\iint_{B(P,r)\cap\Omega} F(Z)^{p/n}\,dZ
&\leq\, C r^{-n}\int_{|x-z|<r}\int_{\varphi(z)}^{\varphi(z)+2r}F(z,s)^{p/n}\, ds dz\\[4pt]
&\leq \,C r\fint_{|x-z|<r}\Big(N_*(F)(z,\varphi(z))\Big)^{p/n}\,  dz\\[4pt]
&\leq\, C \left(\int_{|x-z|<r}
\Big(N_*(F)(z,\varphi(z))\Big)^{p}\,  dz\right)^{1/n}\,\leq \,C 
\|N_*(F)\|_{L^p(\partial\Omega)}^{p/n}\,,
\end{align*}
since the bound \eqref{impr-1} follows, as required.

\smallskip

\noindent{\bf Case 2}.
Now suppose that $Lw=0$ in $\Omega$, and that $\widetilde{N}_*(w)\in L^p(\partial\Omega)$ for 
some $p\in(0,\infty)$. By Moser's sub-mean inequality \eqref{eq2.M}, we have
$\widetilde{N}_*(w)(Q)\approx\|w\|_{L^\infty(Q+\Gamma)}=:N_*(w)(Q)$, 
uniformly for $Q\in\partial\Omega$,
at least if $\beta>0$ is sufficiently small. Under this assumption, 
estimate (\ref{impr-1}) can then be proved as in Case 1, except that invoking 
H\"older's inequality, which was the source of the restriction 
$p\leq 2n/(n+1)$, is unnecessary. This completes the proof of Case 2, since the restriction on the size of $\beta$ is immaterial by Lemma~\ref{L-cones}. 

\smallskip

\noindent{\bf Case 3}.
Finally, suppose that $w=\nabla u$ for some solution $u\in L^2_{1,\,\mathrm{loc}}(\Omega)$ of $Lu=0$ in $\Omega$, and that $\widetilde{N}_*(w)\in L^p(\partial\Omega)$ for some $p\in(0,\infty)$. It is well-known (cf., e.g., \cite{Ke}) that
there exists $q=q(n,\Lambda)>2$ such that 
\[
\left(\fint_{B(X,\,\beta\delta(X))}|w(Y)|^{q}\,dY\right)^{1/q}
\leq C\left(\fint_{B(X,\,2\beta\delta(X))}|w(Y)|^2\,dY\right)^{1/2}.
\]
The proof of \eqref{impr-1} when $0<p<qn/(n+1)$ then proceeds as in Case 1, where Lemma~\ref{L-cones}
is used once more to readjust the size of the balls. 
\end{proof}

\vspace{0.08in}
\vspace{0.08in}

\noindent {\tt S. Hofmann}: Department of Mathematics, 
University of Missouri-Columbia, Columbia, MO 65211, USA, 
{\it e-mail: hofmanns@missouri.edu}

\vskip 0.10in

\noindent {\tt M. Mitrea}: Department of Mathematics, 
University of Missouri-Columbia, Columbia, MO 65211, USA, 
{\it e-mail: mitream@missouri.edu}

\vskip 0.10in

\noindent {\tt A.J.~Morris}: Mathematical Institute, 
University of Oxford, Oxford, OX2 6GG,~UK, 
{\it e-mail: andrew.morris@maths.ox.ac.uk}


\begin{thebibliography}{999}
\small

\bibitem{AAAHK} M.~Alfonseca, P.~Auscher, A.~Axelsson, S.~Hofmann and
S.~Kim, {\it Analyticity of layer potentials and $L^2$ solvability of boundary
value problems for divergence form elliptic equations with complex
$L^\infty$ coefficients}, Adv. Math. {\bf 226} (2011), 4533--4606. 

\bibitem{A} P.~Auscher,  {\it Regularity theorems and heat kernel for elliptic operators},  
J. London Math. Soc. (2) {\bf 54}  (1996),  no. 2, 284--296.

\bibitem{AuscherSurvey} P.~Auscher, {\it On necessary and sufficient
conditions for $L^p$-estimates of Riesz transforms associated with
elliptic operators on $\RR^n$ and related estimates}, 
 Mem. Amer. Math. Soc.  {\bf 186} (2007), no. 871.

\bibitem{AAH} P.~Auscher, A.~Axelsson, and S.~Hofmann,
{\it Functional calculus of Dirac operators and complex perturbations of
Neumann and Dirichlet problems},  J. Funct. Anal. {\bf 255} (2008), 374--448.

\bibitem{AA} P.~Auscher and A.~Axelsson,  {\it  Weighted maximal regularity estimates
and solvability of non-smooth elliptic systems I}, Invent. Math. {\bf 184} (2011),  47--115.

\bibitem{AAM} P.~Auscher, A.~Axelsson, and A.~McIntosh,
{\it Solvability of elliptic systems with square integrable boundary
data}, Ark. Mat. {\bf 48} (2010), 253--287.

\bibitem{AT} P.~Auscher and Ph.~Tchamitchian, {\it Square root problem for divergence operators and related topics}, Ast\'{e}risque {\bf 249} (1998), 1--172.

\bibitem{AHLMcT} P.~Auscher, S.~Hofmann, M.~Lacey, A.~McIntosh, and P.~Tchamitchian, {\it The solution of the Kato Square Root Problem for Second Order Elliptic operators on $\mathbb{R}^n$}, Ann. of Math. (2) {\bf 156} (2002), 633--654.

\bibitem{Bar} A.~Barton, {\it Elliptic partial differential equations with complex coefficients}, Ph.D. Thesis, University of Chicago, 2010.

\bibitem{BM} A.~Barton and S.~Mayboroda, {\it Layer potentials and boundary-value problems for second order elliptic operators with data in Besov spaces}, arXiv:1309.5404.

\bibitem{BK} S.~Blunck and P.~Kunstmann, {\it Weak-type $(p, p)$ estimates for Riesz
transforms}, Math. Z. {\bf 247} no. 1 (2004), 137--148.

\bibitem{Br} R.~Brown, {\it The Neumann problem on Lipschitz domains in Hardy spaces
of order less than one}, Pacific J. Math. {\bf 171} (1995), 389--407.

\bibitem{CFK} L.~Caffarelli, E.~Fabes, and C.~Kenig,  Completely singular elliptic-harmonic measures,  
{\it Indiana Univ. Math. J.} {\bf 30}  (1981), no. 6, 917--924.

\bibitem{CMM} R.R.~Coifman, A.~McIntosh, and Y.~Meyer, {\it L'int\'egrale 
de Cauchy d\'efinit un op\'erateur born\'e sur $L_2$ pour les courbes
lipschitziennes}, Ann. of Math. (2) {\bf 116} (1982), 361--387. 

\bibitem{CW} R.R.~Coifman and G.~Weiss, {\it Extensions of Hardy spaces and their use in analysis}, Bull. Amer. Math. Soc. {\bf 83} (1977), no. 4, 569--645.

\bibitem{D} B.~Dahlberg, {\it Estimates of harmonic measure}, Arch. Rational Mech. Anal. {\bf 65}  (1977), no. 3, 275--288.

\bibitem{DK} B.~Dahlberg and C.~Kenig, {\it Hardy spaces and the $L^p$--Neumann problem for Laplace's equation in a Lipschitz domain}, Ann. of Math. (2) {\bf 125} (1987), 437--465.

\bibitem{DeG} E.~De Giorgi, {\it Sulla differenziabilit\`a e l'analiticit\`a 
delle estremali degli integrali multipli regolari}, Mem. Accad. Sci. Torino. 
Cl. Sci. Fis. Mat. Nat. {\bf 3} (1957), 25--43.

\bibitem{FS} C.~Fefferman and E.M.~Stein, {\it $H^{p}$ spaces of several 
variables}, Acta Math. {\bf 129} (1972), 137--193.

\bibitem{FKP}  R.~Fefferman, C.~Kenig, and J.~Pipher, {\it The theory of weights and the Dirichlet problem for elliptic equations},  Ann. of Math. (2)  {\bf 134}  (1991), no. 1, 65--124. 

\bibitem{Frehse} J.~Frehse, {\it An irregular complex valued solution to a scalar uniformly elliptic equation}, Calc. Var. Partial Differential Equations {\bf 33} (2008), no. 3, 263--266. 

\bibitem{Gi} M.~Giaquinta, {\it Multiple Integrals in the Calculus of Variations and Nonlinear Elliptic
Systems}, Annals of Math. Studies {\bf 105}, Princeton Univ. Press, Princeton, NJ, 1983.

\bibitem{GH} A.~Grau de~la~Herran and S.~Hofmann, {\it Generalized local $Tb$ theorems for square functions, and applications}, arXiv:1212.5870.

\bibitem{HKMP1} S.~Hofmann, C.~Kenig, S.~Mayboroda, and J.~Pipher, {\it Square function/non-tangential maximal estimates and the {D}irichlet problem for non-symmetric elliptic operators}, arXiv:1202.2405. 

\bibitem{HKMP2} S.~Hofmann, C.~Kenig, S.~Mayboroda, and J.~Pipher, {\it The {R}egularity problem for second order elliptic operators with complex-valued bounded measurable coefficients}, arXiv:1301.5209.

\bibitem{HK} S.~Hofmann and S.~Kim, {\it The Green function estimates for strongly elliptic systems of second order}, Manuscripta Math. {\bf 124} (2007), 139--172.

\bibitem{HoMa} S.~Hofmann and J.M.~Martell,  {\it  $L\sp p$ bounds for Riesz transforms and square roots associated with second order elliptic operators}, Publ. Mat. {\bf 47} (2003), no. 2, 497--515.

\bibitem{HMaMo} S.~Hofmann, S.~Mayboroda, and M.~Mourgoglou, {\it $L^p$ and endpoint solvability results for divergence form elliptic equations with complex $L^{\infty}$ coefficients}, preprint.

\bibitem{HMM} S.~Hofmann, M.~Mitrea, and A.J.~Morris, {\it The transmission problem for elliptic operators with $L^\infty$ coefficients}, preprint.

\bibitem{HuS} R.~Hurri-Syrj\"{a}nen, {\it An Improved Poincar\'{e} Inequality}, Proc. Amer. Math. Soc. {\bf 120} (1994), no. 1, 213--222.

\bibitem{I} T.~Iwaniec, {\it The Gehring lemma}, Quasiconformal Mappings and Analysis ({A}nn {A}rbor, {MI}, 1995), Springer, New York, 1998, 181--204.
              
\bibitem{JK1} D.~Jerison and C.~Kenig,  {\it  The Dirichlet problem in nonsmooth domains},  
Ann. of Math. (2) {\bf 113}  (1981), no. 2, 367--382.

\bibitem{JK2} D.~Jerison and C.~Kenig,  {\it The Neumann problem on Lipschitz domains},  
Bull. Amer. Math. Soc. (N.S.)  {\bf 4} (1981), no. 2, 203--207.

\bibitem{KM} N.~Kalton and M.~Mitrea, {\it Stability of Fredholm properties on interpolation scales of quasi-Banach spaces and applications}, Trans. Amer. Math. Soc. {\bf 350} (1998), no.~10, 3837--3901.

\bibitem{Ke} C.E.~Kenig, {\it Harmonic analysis techniques for
second order elliptic boundary value problems}, CBMS Regional
Conference Series in Mathematics, No.~83, AMS, Providence, RI, 1994.

\bibitem{KKPT}  C.~Kenig, H.~Koch, J.~Pipher, and T.~Toro, {\it A new approach to absolute continuity of elliptic measure, with applications to non-symmetric equations}, Adv. Math. {\bf 153} (2000), no.~2, 231--298.

\bibitem{KP} C.E.~Kenig and J.~Pipher, {\it The Neumann problem for elliptic 
equations with nonsmooth coefficients}, Invent. Math. {\bf 113}  (1993),  
no. 3, 447--509.

\bibitem{KR} C.~Kenig and D.~Rule, {\it The regularity and Neumann problems for non-symmetric
elliptic operators}, Trans. Amer. Math. Soc. {\bf 361} (2009), 125--160.

\bibitem{KoS} P. ~Koskela and E.~Saksman, {\it Pointwise characterizations of Hardy-Sobolev functions}, Math. Res. Lett. {\bf 15} (2008), no. 4, 727--744.

\bibitem{KW} D.S.~Kurtz and R.~Wheeden, {\it Results on weighted norm inequalities for multipliers}
Trans. Amer. Math. Soc. {\bf 255} (1979), 343--362.

\bibitem{May} S.~Mayboroda, {\it The connections between Dirichlet, Regularity and Neumann
problems for second order elliptic operators with complex bounded measurable coefficients},
Adv. Math. {\bf 225} (2010), no. 4, 1786--1819.

\bibitem{Me}  N.G.~Meyers, {\it Mean oscillation over cubes and H\"older continuity}, Proc. Amer. Math. Soc. 15 (1964) 717--721.

\bibitem{MNP} V.G.~Maz'ya, S.A.~Nazarov, and B.A.~Plamenevski\u\i, {\it Absence of a De Giorgi-type theorem for strongly elliptic equations with complex coefficients}, Zap. Nauchn. Sem. Leningrad. Otdel. Mat. Inst. Steklov. (LOMI) {\bf 115} (1982), 156--168.

\bibitem{MM} O.~Mendez and M.~Mitrea, {\it The {B}anach envelopes of {B}esov and {T}riebel-{L}izorkin spaces and applications to partial differential equations}, J. Fourier Anal. Appl. {\bf 6} (2000), no.~5, 503--531.

\bibitem{MMMM} D.~Mitrea, I.~Mitrea, M.~Mitrea, and S.~Monniaux, {\it {G}roupoid {M}etrization {T}heory with {A}pplications to {A}nalysis on {Q}uasi-{M}etric {S}paces and {F}unctional {A}nalysis}, Springer, New York, 2013.

\bibitem{Mo} J.~Moser, {\it On Harnack's theorem for elliptic differential operators}, Comm. Pure Appl. Math. {\bf 14} (1961), 577--591.

\bibitem{Na} J.~Nash, {\it Continuity of the solutions of parabolic and elliptic equations}, Amer. J. Math. {\bf 80} (1957), 931--954.

\bibitem{R} A.~Ros\'{e}n, {\it Layer potentials beyond singular integral operators}, arXiv:1210.7582.

\bibitem{SV} A.~Stefanov and G.~Verchota, {\it Optimal solvability for the Dirichlet and Neumann 
problems in dimension two}, Mathematics Faculty Scholarship, Paper 135 (2000), http://surface.syr.edu/mat/135.

\bibitem{St1} E.M.~Stein, {\it Singular Integrals and Differentiability Properties of Functions}, Princteon University Press, Princeton, NJ, 1970.

\bibitem{St2} E.M.~Stein, {\it Harmonic Analysis. Real-Variable Methods, Orthogonality, and Oscillatory Integrals}, Princeton Univ. Press, Princeton, NJ, 1993. 

\bibitem{TW} M.~Taibleson and G.~Weiss, {\it The molecular characterization of certain Hardy spaces.
Representation theorems for Hardy spaces}, Ast\'{e}risque {\bf 77} (1980), 67--149.

\bibitem{V} G.~Verchota, {\it  Layer potentials and boundary value problems for Laplace's equation in Lipschitz domains}, J. Funct. Anal. {\bf 59} (1984), 572--611.

\end{thebibliography}
\end{document}